\documentclass[12 pt, hidelinks, leqno]{amsart} 

\usepackage[top=3cm, bottom=3cm, left=2cm, right=2cm]{geometry}
\usepackage[english]{babel}
\usepackage{hyperref}
\usepackage{amsmath}
\usepackage{amsthm}
\usepackage{amssymb}
\usepackage{ucs}
\usepackage{mathtools}
\usepackage{enumerate}

\usepackage[T1]{fontenc}
\usepackage{color}
\usepackage{tikz}
\usepackage{diagbox}
\usepackage{tikz-cd}
\usepackage{tkz-graph}
\usepackage{relsize,exscale}
\usepackage{mathrsfs}

\AtBeginDocument{
   \def\MR#1{}
}

\theoremstyle{plain}
\newtheorem{theorem}{Theorem}[section]

\newtheorem{mainthm}{Theorem}

\newtheorem{lemma}[theorem]{Lemma}
\newtheorem{proposition}[theorem]{Proposition}
\newtheorem{corollary}[theorem]{Corollary}
\theoremstyle{definition}
\newtheorem{definition}[theorem]{Definition}

\newtheorem{remark}[theorem]{Remark}

\newcommand{\C}{\mathbb{C}}

\newcommand{\Lcal}{\mathcal{L}}

\newcommand{\ot}{\otimes}
\newcommand{\id}{{\rm id}}

\newcommand{\Pol}{{\rm Pol}}
\newcommand{\Irr}{{\rm Irr}}
\newcommand{\Linf}{{\rm L}^\infty}

\newcommand{\Ad}{{\rm Ad}}

\newcommand{\GG}{\mathbb{G}}

\newcommand{\NC}{{\rm NC}}

\newcommand{\Ucal}{\mathcal{U}}

\begin{document}

\title{Free wreath products by finite quantum groups}
\author{Pierre Fima}
\address{Pierre Fima
\newline
Universit\'e Paris Cit\'e, Sorbonne Universit\'e, CNRS, IMJ-PRG, F-75013 Paris, France.}
\email{pierre.fima@imj-prg.fr}

 \author{Yigang Qiu}
\address{Yigang Qiu
\newline
Universit\'e Paris Cit\'e and Sorbonne Universit\'e, CNRS, IMJ-PRG, F-75013 Paris, France.}
\email{yqiu@imj-prg.fr}

\begin{abstract}
We show a natural class of generalized free wreath
products in the sense of Fima and Troupel, namely
$\mathbb G\wr_{*,\Delta_K}K$ with $\mathbb G$ an arbitrary
compact quantum group and $K$ a finite quantum group,
admits a smash-product decomposition and an explicit
formula for all Haar mixed moments separating the
contributions of the two quantum groups.
In the finite group dual case, we further establish
central decompositions, fullness, simplicity and
uniqueness of trace under suitable assumptions.
\end{abstract}

\maketitle

\section{Introduction}
\noindent 
The theory of compact quantum groups, initiated by Woronowicz
\cite{wor87,wor88,wor98}, provides a framework in which
classical symmetries are replaced by genuinely noncommutative
ones. A fundamental example is Wang's quantum permutation
group $S_N^+$, the universal compact quantum group acting
on an $N$-point set.

\medskip

\noindent Free wreath products were introduced by Bichon
\cite{Bi04} in connection with quantum symmetries of finite
graphs. If $X$ is a finite connected graph, the quantum
symmetry of the disjoint union $X^{\sqcup N}$ is described
by a free wreath product of the form
\[
\operatorname{Aut}^+(X)\wr_* S_N^+.
\]
Thus free wreath products arise naturally from quantum
automorphism groups of finite structures. Since Bichon's
work, several extensions have been developed
\cite{pit14,FP16,TW18,FS18,Fr22,FT25}, with applications
to representation theory, free probability and operator
algebras.

\medskip

\noindent Brannan, Gromada, Matsuda, Skalski and
Wasilewski \cite{BGMSW25} establish a quantum analogue
of Frucht's theorem and show that every finite quantum
group $K$ is the common quantum automorphism group of
finitely many quantum Cayley graphs on $C(K)$.
For $K=\widehat{\Lambda}$ with $\Lambda$ finite
and non-abelian, they further obtain a single
rigid quantum Cayley graph with non-trivial adjacency
operator $d_\Lambda$, satisfying
\[
\operatorname{Aut}^+
(C^*(\Lambda),\tau_\Lambda,d_\Lambda)
\simeq \widehat{\Lambda}.
\]
These results provide a natural geometric motivation
for free wreath products by finite quantum groups.

\medskip

\noindent A general framework for such free wreath
products was introduced by Fima and Troupel \cite{FT25}.
Given compact quantum groups $\mathbb G,\mathbb H$
and an action
$\beta:\mathbb H\curvearrowright(B,\psi)$,
they define the free wreath product
$\mathbb G\wr_{*,\beta}\mathbb H$.
Within this framework, we identify and investigate
a distinguished family arising from the canonical
comultiplication actions of finite quantum groups:
\[
\mathbb W(\mathbb G,K)
:=\mathbb G\wr_{*,\beta_\Delta}K,
\]
where $\mathbb G$ is an arbitrary compact quantum
group, $K$ is a finite quantum group, and
$\beta_\Delta:K\curvearrowright(C(K),h_K)$
is the canonical comultiplication action
$\beta_\Delta=\Delta_K$.
The results of \cite{BGMSW25} provide a natural
quantum-graph interpretation of these actions,
while the finite group dual case admits the
particularly concrete realization in terms of
a single rigid quantum Cayley graph.

\medskip

\noindent This connection can be made precise within
the framework of Fima and Pittau \cite{FP16}.
Their construction
$\mathbb G\wr_*\operatorname{Aut}^+(B,\psi,d)$,
introduced in \cite[Definition~2.6]{FP16},
allows the quantum automorphism group to preserve
a prescribed intertwiner $d$, and hence incorporates
non-trivial adjacency relations.
Fima and Troupel \cite[Proposition~3.12]{FT25}
identify their construction with that of Fima
and Pittau for $d=1$. In the appendix, we extend
this identification to non-trivial admissible
intertwiners $d$, proving $\mathbb G\wr_{*,\beta_d}
\operatorname{Aut}^+(B,\psi,d)
\simeq
\mathbb G\wr_*
\operatorname{Aut}^+(B,\psi,d).$ Consequently, for every finite non-abelian group
$\Lambda$, we obtain $$\mathbb W(\mathbb G,\widehat{\Lambda})
\simeq
\mathbb G\wr_*
\operatorname{Aut}^+
(C^*(\Lambda),\tau_\Lambda,d_\Lambda).$$
Thus the finite group dual case of our family
is naturally realized as a free wreath product
associated with quantum graph symmetries in the
sense of \cite{FP16}, including genuinely
non-trivial adjacency data.

\medskip

Despite this natural geometric interpretation, the
operator-algebraic structure of this family remains
largely unexplored. Indeed, the main operator-algebraic
results of \cite{FP16} concern free wreath products by
$\operatorname{Aut}^+(B,\psi)$ without additional adjacency
relations, whereas the factoriality, primeness and
absence of Cartan subalgebras results of \cite{FT25}
rely on $2$-ergodicity and infiniteness assumptions.
In our setting, $K$ is finite and the comultiplication
action is ergodic but not $2$-ergodic when
$\dim C(K)>2$; see \cite[Definition~2.16,
Remark~2.17 and Example~2.18(2)]{FT25}.

\medskip

\noindent A particularly striking feature of this family
is that its Haar state admits a simple and explicit
formula for mixed moments of arbitrary order.
The formula separates the contributions of the two
quantum groups: the left-hand side enters through
the free cumulants of $h_{\mathbb G}$, while the
finite quantum group $K$ contributes only through
its Haar state $h_K$. These contributions are combined
by non-crossing partitions, yielding a closed
combinatorial expression without Weingarten
matrix inversion.

\medskip

\noindent A particularly striking feature of this family
is that its Haar state admits a simple, explicit and
factorized formula for mixed moments of arbitrary order.
While \cite{LT16} provides a Weingarten formula for
mixed moments of matrix coefficients in free wreath
products by quantum permutation groups, and
\cite{FT24,FT25} describe the Haar state through
reduced-word conditions and conditional expectations,
respectively, our formula directly expresses all
mixed moments of the canonical generators as a sum
over non-crossing partitions. More importantly, it
separates the contributions of the two quantum groups:
$\mathbb G$ enters only through the free cumulants
of its Haar state, whereas $K$ enters only through
its Haar moments. This yields a closed combinatorial
formula without Weingarten matrix inversion.

\medskip

\begin{mainthm}\label{ThmA}
Let $\mathbb G$ be an arbitrary compact quantum group
and $K$ a finite quantum group. Set $D=\dim C(K)$,
and let $\mathcal B_K$ be an orthonormal basis of
$L^2(K,h_K)$.

The free wreath product admits a smash-product
decomposition
\[
C(\mathbb W(\mathbb G,K))
\simeq \Ucal\# C(K),
\]
where $\Ucal$ is generated by the canonical elements
$\nu_x(a)$, with $x\in\mathcal B_K$ and
$a\in C(\mathbb G)$. In particular, every product
of generators can be written as a finite linear
combination of elements of the form
\[
\nu_{x_1}(a_1)\cdots\nu_{x_n}(a_n)b,
\qquad b\in C(K).
\]

The Haar state $h_{\mathbb W}$ is determined by
the following explicit formula for all such
mixed moments: for every $n\geq 1$,
$a_1,\ldots,a_n\in C(\mathbb G)$,
$x_1,\ldots,x_n\in\mathcal B_K$, and $b\in C(K)$,
\[
\begin{aligned}
&h_{\mathbb W}\bigl(
\nu_{x_1}(a_1)\cdots\nu_{x_n}(a_n)b
\bigr)\\
&\quad=
h_K(b)
\sum_{\pi\in\NC(n)}
\frac{\kappa_\pi^{h_{\mathbb G}}
(a_1,\ldots,a_n)}{D^{n-|\pi|}}
\prod_{\substack{V=\{v_1<\cdots<v_r\}\\V\in\pi}}
h_K\bigl((x_{v_1}\cdots x_{v_r})^*\bigr).
\end{aligned}
\]
For $n=0$, the formula reduces to
$h_{\mathbb W}(b)=h_K(b)$.
\end{mainthm}

\noindent The formula in Theorem \ref{ThmA} is remarkably
simple: each non-crossing partition contributes a product
of free cumulants of $h_{\mathbb G}$ and Haar moments of
$K$, separating the contributions of the two quantum groups.

\medskip

\noindent Beyond the explicit computation of the Haar state,
our structural description also allows us to investigate
the associated operator algebras. In the finite group
dual case, the smash-product decomposition becomes a
crossed-product decomposition. Under suitable assumptions,
the induced action of $\Lambda/\mathcal Z(\Lambda)$ is outer.
This, together with the Haar state formula, allows us
to obtain explicit central decompositions of
$C_r(\mathbb W)$ and $\Linf(\mathbb W)$, with simple
unique-trace and full $\mathrm{II}_1$-factor summands,
respectively.

 \medskip

\begin{mainthm}\label{ThmB}
Let $Z=\mathcal Z(\Lambda)$, $Q=\Lambda/Z$, and $p_\chi=|Z|^{-1}\sum_{z\in Z}\overline{\chi(z)}u_z$, $\chi\in\widehat Z$.
Then

$$
\Linf(\mathbb W)=\bigoplus_{\chi\in\widehat Z}\Linf(\mathbb W)p_\chi,
\qquad
C_r(\mathbb W)=\bigoplus_{\chi\in\widehat Z}C_r(\mathbb W)p_\chi.
$$

Moreover, for any section $s:Q\to\Lambda$ with $s(e)=e$, setting $c_\chi(q,r)=\chi\bigl(s(q)s(r)s(qr)^{-1}\bigr)$ for $q,r\in Q$, one has

$$
\Linf(\mathbb W)p_\chi\cong \Ucal''\rtimes_{\bar\alpha,c_\chi}Q,
\qquad
C_r(\mathbb W)p_\chi\cong \Ucal_r\rtimes_{r,\bar\alpha,c_\chi}Q,
$$

where $\bar\alpha$ is the action induced by $\alpha$.

If $\mathbb G$ is of Kac type, $\Linf(\mathbb G)$ is a diffuse factor, and $\Lambda$ is non-trivial, then these decompositions are the central decompositions: each $\Linf(\mathbb W)p_\chi$ is a full $\mathrm{II}*1$-factor and each $C_r(\mathbb W)p*\chi$ is simple with a unique trace. Consequently,

$$
\mathcal Z(\Linf(\mathbb W))\cong {\rm L}(Z)
\quad\text{and}\quad
\mathcal Z(C_r(\mathbb W))\cong C^*(Z).
$$

In particular, if $Z={e}$, then $C_r(\mathbb W)$ is simple with a unique trace and $\Linf(\mathbb W)$ is a full $\mathrm{II}_1$-factor.
\end{mainthm}

\medskip

\subsection*{Acknowledgements} Pierre Fima acknowledges the support of CNRS IEA GAOA and ANR-25-CE40-5010 CroCQG.
The authors would like to thank Sheng Yin for
helpful conversations.

\section*{Declaration of AI Use}

AI tools were used to assist with \LaTeX{} editing, 
language polishing, and the identification of relevant 
references. The authors take full responsibility for 
the content and correctness of this work.

\section{Preliminaries}

\noindent The same symbol $\ot$ will be used to denote the tensor product of Hilbert spaces as well as the minimal tensor product of C*-algebras.

\subsection{Generalities on compact quantum groups}

We recall some basic facts about compact quantum groups (CQGs)
\cite{wor87,wor88,wor98}. A CQG $\mathbb G$ is a pair
$(A,\Delta)$, where $A$ is a unital C*-algebra and
$\Delta:A\to A\ot A$ is a unital *-homomorphism satisfying
\[
\begin{gathered}
(\Delta\ot\id)\Delta=(\id\ot\Delta)\Delta,\\
\overline{\mathrm{span}}\,\Delta(A)(1\ot A)
=\overline{\mathrm{span}}\,\Delta(A)(A\ot1)=A\ot A.
\end{gathered}
\]
Here $\ot$ denotes the minimal C*-tensor product.
We use the universal completion $C(\mathbb G)$.
The unique Haar state $h_{\mathbb G}$ satisfies
\[
(h_{\mathbb G}\ot\id)\Delta_{\mathbb G}
=(\id\ot h_{\mathbb G})\Delta_{\mathbb G}
=h_{\mathbb G}(\cdot)1.
\]
Writing $\pi_h$ for its GNS representation, we set
$C_r(\mathbb G)=\pi_h(C(\mathbb G))$ and
$\Linf(\mathbb G)=\pi_h(C(\mathbb G))''$.
The Haar state is faithful on $C_r(\mathbb G)$
and extends to a faithful normal state on $\Linf(\mathbb G)$.

A finite-dimensional unitary representation is a unitary
$U=(u_{ij})\in B(H_U)\ot C(\mathbb G)$ satisfying
$(\id\ot\Delta_{\mathbb G})(U)=U_{12}U_{13}$, equivalently
$\Delta_{\mathbb G}(u_{ij})=\sum_k u_{ik}\ot u_{kj}$.
By Peter--Weyl theory,
\[
\Pol(\mathbb G)=
\bigoplus_{\alpha\in\Irr(\mathbb G)}
\mathrm{span}\{u_{ij}^{\alpha}:1\leq i,j\leq d_\alpha\}
\]
is a Hopf $*$-algebra, norm dense in $C(\mathbb G)$
and $C_r(\mathbb G)$, and weakly dense in $\Linf(\mathbb G)$.
Its counit $\varepsilon$ and antipode $S$ satisfy
$$
\varepsilon(u_{ij}^\alpha)=\delta_{ij},\quad
S(u_{ij}^\alpha)=(u_{ji}^\alpha)^*,$$
$$m(S\ot\id)\Delta=m(\id\ot S)\Delta=\varepsilon(\cdot)1,(\varepsilon\ot\id)\Delta=(\id\ot\varepsilon)\Delta=\id.$$

Tensor products on $\Pol(\mathbb G)$ are algebraic.

A finite quantum group $\mathbb K$ has finite-dimensional
$C(\mathbb K)$ \cite{vandaele97}. Then
$\Pol(\mathbb K)=C(\mathbb K)=C_r(\mathbb K)
=\Linf(\mathbb K)$, and $\mathbb K$ is of Kac type.
Thus $h_{\mathbb K}$ is a faithful trace,
$S_{\mathbb K}$ is an involutive *-antiautomorphism,
and $\varepsilon_{\mathbb K}$ is a *-homomorphism.
Moreover, $h_{\mathbb K}\circ S_{\mathbb K}=h_{\mathbb K}$.

The dual $C(\mathbb K)^*$ is again a finite quantum group
$\widehat{\mathbb K}$, with convolution product
$(\omega\eta)(x)=(\omega\ot\eta)\Delta_{\mathbb K}(x)$
and coproduct
$\Delta_{\widehat{\mathbb K}}(\omega)(x\ot y)=\omega(xy)$. Its irreducible unitary
representations correspond to the irreducible
$*$-representations of $C(K)$.
Consequently,
\[
C(K)\cong\bigoplus_{u\in\Irr(\widehat K)}
M_{N_u}(\C),\qquad
\dim C(K)=\sum_{u\in\Irr(\widehat K)}N_u^2,
\]
where $N_u$ is the dimension of the corresponding
irreducible $*$-representation
$\pi_u:C(K)\to M_{N_u}(\C)$.
The Haar state is given by
\[
h_K(x)=\frac{1}{\dim C(K)}
\sum_{u\in\Irr(\widehat K)}
N_u\operatorname{Tr}(\pi_u(x)).
\]

\begin{lemma}\label{LemOneSidedHaar}
Let $\mathbb G$ be a compact quantum group and
$\varphi$ a state on $C(\mathbb G)$.
Then $\varphi=h_{\mathbb G}$ if and only if
$\varphi$ is left invariant or right invariant.
\end{lemma}

This follows from the uniqueness of the Haar state;
see \cite{VD95}.

\subsection{Smash products and universal crossed products}

Let $H=\Pol(\mathbb G)$ and $B$ be a unital C*-algebra
with a left $H$-module $*$-algebra structure
$\triangleright$ \cite[Definition~5.1]{decommer17}.
Thus, $1_H\triangleright x=x$,
$(hk)\triangleright x=h\triangleright(k\triangleright x)$,
and, using Sweedler notation with implicit summation,
\[
\begin{gathered}
h\triangleright(xy)=(h_{(1)}\triangleright x)
(h_{(2)}\triangleright y),\\
h\triangleright 1=\varepsilon(h)1,\qquad
(h\triangleright x)^*=S(h)^*\triangleright x^*.
\end{gathered}
\]
If $\mathbb G$ is finite, such structures correspond
bijectively to right actions
$\alpha:B\to B\otimes\Pol(\widehat{\mathbb G})$ via
$h\triangleright x=(\id\otimes h)\alpha(x)$, where
$H=\Pol(\widehat{\mathbb G})^*$
\cite[Example~5.4]{decommer17}.
In particular, the adjoint action of $H$ on itself,
given by $h\triangleright_{\mathrm{ad}} k
=h_{(1)}kS(h_{(2)})$, makes $H$ a left
$H$-module $*$-algebra
\cite[Example~5.5]{decommer17}.

The algebraic smash product $B\#_{\mathrm{alg}}H$
is $B\otimes_{\mathrm{alg}}H$ with
\[
\begin{gathered}
(x\#h)(y\#k)=x(h_{(1)}\triangleright y)\#h_{(2)}k,\\
(x\#h)^*=(h_{(1)}^*\triangleright x^*)\#h_{(2)}^*.
\end{gathered}
\]
Its universal C*-completion, denoted by $B\#_{\mathrm u}H$,
exists and contains $B\#_{\mathrm{alg}}H$ faithfully
\cite[Definition~5.8, Lemma~5.12 and Corollary~5.17]
{decommer17}.

For a finite group $\Lambda$ acting on $B$ by
$*$-automorphisms $\alpha:\Lambda\curvearrowright B$,
taking $H=\mathbb C[\Lambda]=\Pol(\widehat{\Lambda})$
and $r\triangleright x=\alpha_r(x)$ yields
\[
B\#_{\mathrm u}\mathbb C[\Lambda]
\cong B\rtimes_{\alpha,\mathrm u}\Lambda
=B\rtimes_{\alpha,\mathrm r}\Lambda,
\]
generated by $B$ and unitaries $(u_r)_{r\in\Lambda}$
satisfying $u_rxu_r^*=\alpha_r(x)$ and $u_ru_s=u_{rs}$.

\subsection{Von Neumann crossed products}

Let $(M,\tau)$ be a finite von Neumann algebra and
$\alpha:\Lambda\curvearrowright M$ a trace-preserving
action of a finite group. The \emph{von Neumann
crossed product} $M\rtimes_\alpha\Lambda$ is the
von Neumann algebra generated by $M$ and the
canonical unitaries $(u_r)_{r\in\Lambda}$ in the
regular covariant representation. Equivalently,
\[
M\rtimes_\alpha\Lambda
=\bigl(\pi_\alpha(M)\cup
\{u_r:r\in\Lambda\}\bigr)''.
\]
Since $\Lambda$ is finite, it coincides with the
C*-crossed product and every element has a unique
expansion $\sum_{r\in\Lambda}x_ru_r$.
The canonical faithful normal conditional expectation
and trace are
\[
E_M\Big(\sum_r x_ru_r\Big)=x_e,
\qquad \widetilde{\tau}=\tau\circ E_M.
\]

An automorphism of $M$ is \emph{outer} if it is not
inner, and $\alpha$ is outer if $\alpha_r$ is outer
for every $r\neq e$. An inclusion of factors
$P\subset R$ is \emph{irreducible} if
$P'\cap R=\mathbb C1$, and has \emph{finite index}
if its Jones index $[R:P]$ is finite.
If $M$ is a factor and $\alpha$ is outer, then
$M\rtimes_\alpha\Lambda$ is a factor and
\[
M'\cap(M\rtimes_\alpha\Lambda)=\mathbb C1,
\qquad [M\rtimes_\alpha\Lambda:M]=|\Lambda|.
\]
A finite factor $M$ is \emph{full} if
$\operatorname{Inn}(M)$ is closed in
$\operatorname{Aut}(M)$ for the $u$-topology.
If $M$ is full and $\alpha$ is outer, then
$M\rtimes_\alpha\Lambda$ is also full.

\begin{lemma}\cite[Lemma~11.2.6]{jonesVN2015}\label{LemJonesInner}
Let $M$ be a factor and $\beta\in\operatorname{Aut}(M)$.
If there exists $0\neq x\in M$ such that
\[
ax=x\beta(a),\qquad a\in M,
\]
then $\beta$ is inner.
\end{lemma}
 .

\begin{lemma}\label{LemCenterCrossedProduct}
Let $M$ be a factor, put $Z=\mathcal Z(\Lambda)$,
and assume $Z\subset\ker(\alpha)$ and that
$\alpha_r$ is outer for every $r\in\Lambda\setminus Z$.
Then
\[
M'\cap(M\rtimes_\alpha\Lambda)
=\mathcal Z(M\rtimes_\alpha\Lambda)
=L(Z)=\operatorname{span}\{u_z:z\in Z\}.
\]
\end{lemma}

\begin{proof}
Since $Z\subset\ker(\alpha)\cap\mathcal Z(\Lambda)$,
we have $L(Z)\subset\mathcal Z(M\rtimes_\alpha\Lambda)$.
Conversely, let
$x=\sum_{r\in\Lambda}x_ru_r\in M'\cap
(M\rtimes_\alpha\Lambda)$. Comparing coefficients gives
\[
ax_r=x_r\alpha_r(a),\qquad a\in M.
\]
By Lemma~\ref{LemJonesInner}, $x_r\neq0$ implies
that $\alpha_r$ is inner, hence $r\in Z$.
For $r\in Z$, we have $\alpha_r=\id$, so
$x_r\in\mathcal Z(M)=\mathbb C1$.
Thus $x\in L(Z)$, proving the result.
\end{proof}

\begin{lemma}\label{LemCenterCrossedProduct}
Let $M$ be a factor, put $Z=\mathcal Z(\Lambda)$, and assume
$Z\subset\ker(\alpha)$ and that $\alpha_r$ is outer for every
$r\in\Lambda\setminus Z$. Then
\[
\mathcal Z(M\rtimes_\alpha\Lambda)
=L(Z)=\operatorname{span}\{u_z:z\in Z\}.
\]
\end{lemma}

\begin{proof}
The inclusion $L(Z)\subset\mathcal Z(M\rtimes_\alpha\Lambda)$ follows from
$Z\subset\ker(\alpha)$ and $Z\subset\mathcal Z(\Lambda)$. Conversely, write
$x=\sum_{r\in\Lambda}x_ru_r\in\mathcal Z(M\rtimes_\alpha\Lambda)$. For
$a\in M$, one has
\[
ax_r=x_r\alpha_r(a).
\]
If $x_r\neq0$, then $x_r^*x_r$ commutes with $\alpha_r(M)=M$ and
$x_rx_r^*$ commutes with $M$. Since $M$ is a factor, both are scalar.
Hence $x_r$ is a non-zero scalar multiple of a unitary, and the above
relation implies that $\alpha_r$ is inner. Therefore $r\in Z$. For
$r\in Z$, since $\alpha_r=\id$, the same relation gives
$x_r\in\mathcal Z(M)=\mathbb C1$.
\end{proof}

\begin{corollary}\label{CorCenterDecompositionCrossedProduct}\label{CorTwistedSummandsIIoneFactors}
Under the hypotheses of Lemma~\ref{LemCenterCrossedProduct}, set
$p_\chi=|Z|^{-1}\sum_{z\in Z}\overline{\chi(z)}u_z$, $\chi\in\widehat Z$.
Then
$$
M\rtimes_\alpha\Lambda
=\bigoplus_{\chi\in\widehat Z}(M\rtimes_\alpha\Lambda)p_\chi
$$
is its central decomposition. If $M$ is a $\mathrm{II}_1$-factor, every
summand is a $\mathrm{II}_1$-factor.
\end{corollary}

\begin{proof}
\medskip

\noindent Fourier transform gives
$L(Z)=\bigoplus_{\chi\in\widehat Z}\mathbb Cp_\chi$. By
Lemma~\ref{LemCenterCrossedProduct}, the $p_\chi$ are therefore the minimal
central projections of $M\rtimes_\alpha\Lambda$, and the displayed direct
sum is its central decomposition.

If $M$ is a $\mathrm{II}_1$-factor, then $M\rtimes_\alpha\Lambda$ is finite.
Moreover, the map
$x\mapsto xp_\chi$
from $M$ into $(M\rtimes_\alpha\Lambda)p_\chi$ is injective by uniqueness of
Fourier coefficients. Hence every summand is infinite-dimensional. Since
each summand is a finite factor, it is a $\mathrm{II}_1$-factor.
\end{proof}

\medskip

\noindent Let $Q=\Lambda/Z$ and let $s:Q\to\Lambda$ be a section with $s(e)=e$.  For
$\chi\in\widehat Z$ and $q,r\in Q$, put $c_\chi(q,r)=\chi(s(q)s(r)s(qr)^{-1})$.
This is a normalized scalar $2$-cocycle. If
$\beta:Q\curvearrowright M$ is an action and $c$ a normalized scalar
$2$-cocycle, $M\rtimes_{\beta,c}Q$ denotes the von Neumann algebra generated
by $M$ and unitaries $(v_q)_{q\in Q}$ satisfying
$v_qxv_q^*=\beta_q(x)$ and $v_qv_r=c(q,r)v_{qr}$.

\begin{lemma}\label{LemCentralSummandTwistedCrossedProduct}
Assume $Z\subset\ker(\alpha)$ and let $\bar\alpha:Q\curvearrowright M$ be the
induced action.  For every $\chi\in\widehat Z$,
$$
\Phi_\chi:M\rtimes_{\bar\alpha,c_\chi}Q\longrightarrow
(M\rtimes_\alpha\Lambda)p_\chi,
\quad
\Phi_\chi(x)=xp_\chi,\quad
\Phi_\chi(v_q)=u_{s(q)}p_\chi,
$$
is a normal $*$-isomorphism.
\end{lemma}

\begin{proof}Since $u_zp_\chi=\chi(z)p_\chi$, the displayed formulas preserve the twisted
covariance relations and define a normal $*$-homomorphism.  It is surjective
because every $r\in\Lambda$ has a unique expression $r=s(q)z$.  If
$\sum_{q\in Q}x_qv_q$ belongs to its kernel, expansion of $p_\chi$ gives
$$
\sum_{q\in Q}\sum_{z\in Z}
|Z|^{-1}\overline{\chi(z)}x_qu_{s(q)z}=0.
$$
Uniqueness of Fourier coefficients gives $x_q=0$ for every $q\in Q$.
\end{proof}

\begin{remark}\label{RmkCrossed}The same proof gives
$(A\rtimes_{r,\alpha}\Lambda)p_\chi\cong A\rtimes_{r,\bar\alpha,c_\chi}Q$
for every unital C*-algebra $A$ and every action of the finite group
$\Lambda$ with $Z\subset\ker(\alpha)$.
\end{remark}

\begin{corollary}\label{CorTwistedCrossedProductFull}
Let $M$ be a full $\mathrm{II}_1$-factor, let
$\beta:Q\curvearrowright M$ be an outer action of a finite group, and let
$c$ be a normalized scalar $2$-cocycle.  Then
$M\rtimes_{\beta,c}Q$ is a full $\mathrm{II}_1$-factor.
\end{corollary}

\begin{proof}The canonical inclusion
$M\subset M\rtimes_{\beta,c}Q$ is irreducible by outerness and has Jones
index $|Q|$.  The conclusion follows from
Theorem~\cite[Proposition~1.11]{PiPo86}.
\end{proof}

\noindent Let $A$ be a unital C*-algebra with a trace $\varphi$.  A cocycle action
$(\alpha,u)$ of $Q$ on $A$ is called $\varphi$-outer if, for every
$q\in Q\setminus\{e\}$, the normal extension of $\alpha_q$ to $\pi_\varphi(A)''$ is outer.  We shall use
\cite[Theorem~2]{Bedos96}: if $A$ is simple with unique trace $\varphi$ and
the cocycle action is $\varphi$-outer, then the reduced twisted crossed
product is simple with unique trace $\varphi\circ E_A$.

\subsection{Free probability}

\begingroup
\noindent Let $(\mathcal A,\varphi)$ be a non-commutative probability space.  For
$\pi\in\NC(n)$ and a block $V=\{i_1<\cdots<i_k\}$, put
$\varphi_V[a_1,\ldots,a_n]=\varphi(a_{i_1}\cdots a_{i_k})$ and
$$\varphi_\pi[a_1,\ldots,a_n]=\prod_{V\in\pi}\varphi_V[a_1,\ldots,a_n].$$

\noindent The free cumulants are the multilinear functionals
$$
\kappa_\pi[a_1,\ldots,a_n]
=\sum_{\substack{\sigma\in\NC(n)\\ \sigma\leq\pi}}
\varphi_\sigma[a_1,\ldots,a_n]\mu(\sigma,\pi),
\qquad
\kappa_n=\kappa_{1_n}.
$$

\noindent These definitions and the following moment--cumulant formula are
\cite[Definition~11.3 and Proposition~11.4]{nica2006lectures}.

\begin{theorem}\label{ThmCumulant}
For $a_1,\ldots,a_n\in\mathcal A$,
$$
\kappa_\pi[a_1,\ldots,a_n]
=\prod_{V\in\pi}\kappa_{|V|}(a_{i_1},\ldots,a_{i_k})
\quad\text{and}\quad
\varphi(a_1\cdots a_n)=\sum_{\pi\in\NC(n)}
\kappa_\pi[a_1,\ldots,a_n],
$$
where $V=\{i_1<\cdots<i_k\}$ in the first formula.
\end{theorem}

\begin{theorem}[{\cite[Theorem~11.16]{nica2006lectures}}]\label{VanishingCumulant}
Unital subalgebras $(\mathcal A_i)_{i\in I}$ are freely independent if and
only if every mixed free cumulant vanishes: for $n\geq2$ and
$a_j\in\mathcal A_{i(j)}$, one has $\kappa_n(a_1,\ldots,a_n)=0$
whenever $i(1),\ldots,i(n)$ are not all equal.
\end{theorem}

\noindent If $p\in\mathcal A$ is a projection with $\varphi(p)\ne0$, write
$\varphi^p=\varphi(p)^{-1}\varphi|_{p\mathcal A p}$ and denote the associated
cumulants by $\kappa^p$.  A partition $\pi\in\NC(n)$ couples
$\vec i=(i_1,\ldots,i_n)$ and $\vec j=(j_1,\ldots,j_n)$ cyclically if, for
every block $\{r_1<\cdots<r_k\}$ of $\pi$,
$j_{r_\ell}=i_{r_{\ell+1}}$, where $r_{k+1}=r_1$.

\begin{theorem}[{\cite[Theorem~14.18]{nica2006lectures}}]\label{Thm14.18}
Let $(e_{ij})_{1\leq i,j\leq d}\subset\mathcal A$ be matrix units such that
$\varphi(e_{ij})=\delta_{ij}/d$, and let $a_1,\ldots,a_m$ be free from the
algebra generated by these matrix units.  Put $p=e_{11}$ and
$a_{ij}^{(r)}=e_{1i}a_re_{j1}$.  Then
$$
\kappa_n^p\bigl(a_{i_1j_1}^{(r_1)},\ldots,
a_{i_nj_n}^{(r_n)}\bigr)
=
\begin{cases}
d^{1-n}\kappa_n(a_{r_1},\ldots,a_{r_n}),
&j_k=i_{k+1}\ (1\leq k\leq n),\ i_{n+1}=i_1,\\
0,&\text{otherwise.}
\end{cases}
$$
\end{theorem}

\endgroup

\section{Universal quantum homomorphisms}

\noindent In this section we provide a new approach to universal quantum homomorphisms and, in the case of universal quantum homomorphisms from $C(\mathbb G)$ to $C^*(\Lambda)$, where $\mathbb G$ is a compact quantum group and $\Lambda$ is a finite group, we give an explicit formula for the canonical state constructed in \cite{FMP24}.

\subsection{Generalities on universal quantum homomorphisms}

Let $A$, $B$ be two unital C*-algebras. Whenever $B$ is finite dimensional there exists a unique, up to a canonical isomorphism, pair $(\rho,\Ucal)$ of a unital C*-algebra $\Ucal$ and a unital $*$-homomorphism $\rho\,:\,A\rightarrow B\ot\Ucal$ such that, for any pair $(\rho',\Ucal')$ of a unital C*-algebra $\Ucal'$ and a unital $*$-homomorphism $\rho'\,:\, A\rightarrow B\ot\Ucal'$ there exists a unique unital $*$-homomorphism $\pi\,:\,\Ucal\rightarrow\Ucal'$ such that $(\id\ot\pi)\rho=\rho'$. Following \cite{FMP24}, the morphism $\rho$ is called \textit{the universal quantum homomorphism from $A$ to $B$}. Both the morphism $\rho$ and the C*-algebra $\Ucal$ are studied in detail in \cite{FMP24}.

\medskip

\noindent Write $B=\bigoplus_{\kappa=1}^KM_{N_\kappa}(\C)$, with matrix units $(e^\kappa_{ij})_{\kappa,i,j}$ and $\chi_\kappa\,:\, B\rightarrow M_{N_\kappa}(\C)$ the canonical surjection. Given a pair of states $\omega\in A^*$ and $\mu\in B^*$, a specific state $\widetilde{\omega}\in\Ucal^*$ is constructed in \cite[Section 3.1.3]{FMP24}. In what follows we take the state $\mu=\frac{1}{{\rm dim}(B)}\sum_{\kappa=1}^KN_\kappa{\rm Tr}_\kappa$, where ${\rm Tr}_\kappa$ is the unique trace on $M_{N_\kappa}(\C)$ such that ${\rm Tr}_\kappa(1)=N_\kappa$. Let $(H,\pi,\xi)$ be the GNS construction of $\omega$. We still denote by $\omega:=\langle\cdot \xi,\xi\rangle$ the GNS-faithful state on the C*-algebra $\pi(A)\subset\Lcal(H)$ and we define $A_r:=\pi(A)$ and $A'':=\pi(A)''$. Let $(\widetilde{\pi},\widetilde{H},\widetilde{\xi})$ be the GNS construction of $\widetilde{\omega}$, $\Ucal_r:=\widetilde{\pi}(\Ucal)$ be the reduced C*-algebra given by the GNS construction of $\widetilde{\omega}$ and $\Ucal'':=\Ucal_r''$ the von Neumann algebra generated by $\Ucal_r$ in the GNS representation of $\widetilde{\omega}$. It is shown in \cite{FMP24} that, provided that the normal extension of $\omega$ to $A''$ is faithful, the map $\rho\,:\,A\rightarrow B\ot \Ucal$ has a unique faithful normal extension $\rho''\,:\,A''\rightarrow B\ot \Ucal''$ satisfying $\rho''(A_r)\subset B\ot\Ucal_r$.

\begin{lemma}\label{LemFactor}
If $A''$ is a diffuse factor and the normal extension of $\omega$ to $A''$ is faithful, then:
$$\rho''(A'')'\cap(B\ot \Ucal'')\subset\mathcal{Z}(B)\ot 1.$$
\end{lemma}

\begin{proof}

\medskip

\noindent We first record a consequence of \cite[Proposition~3.3]{Ue13}. Let
$$(M,E)=(P,E_P)*_D(R,E_R)$$
be an amalgamated free product, where $D$ is finite dimensional, $P$ and $R$ are $\sigma$-finite, and let
$Q\subset P$ be a diffuse $\sigma$-finite von Neumann subalgebra with a faithful normal
conditional expectation $E_Q:P\to Q$. By
\cite[Lemma~2.1]{HU16b}, there exists a faithful normal state $\psi$ on
$Q$ such that the centralizer $Q_\psi$ is diffuse. The state
$\varphi=\psi\circ E_Q$ is faithful and normal on $P$. For all
$c\in Q_\psi$ and $y\in P$, one has
$\varphi(cy)=\psi(cE_Q(y))=\psi(E_Q(y)c)=\varphi(yc)$,
so $Q_\psi\subset P_\varphi$. Choose a net
$(v_i)_i\subset\mathcal U(Q_\psi)$
converging ultraweakly to $0$. Since $D$ is finite dimensional, for all
$a,b\in P$ one has $E_P(b^*v_i a)\longrightarrow0$
in norm, hence strongly. Therefore, \cite[Proposition~3.3]{Ue13}, applied
with the strongly dense $*$-subalgebra $P\subset P$, gives
$Q_\psi'\cap M=Q_\psi'\cap P$. Since $Q'\subset Q_\psi'$, it follows that
$Q'\cap M\subset P$.

\medskip

\noindent Assume first that $K=1$, so that $B=M_N(\C)$. By \cite{FMP24}, there is
a state-preserving identification $M_N(\C)\ot\Ucal''=M_N(\C)*A''$
under which $\rho''$ is the inclusion of $A''$. Regarding this free
product as $A''*M_N(\C)$ and applying the preceding localization result
with $P=Q=A''$ and
$D=\C$ gives
$\rho''(A'')'\cap(M_N(\C)\ot\Ucal'')\subset\rho''(A'')$.
Since $A''$ is a factor, the left-hand side is contained in
$\rho''(A'')'\cap\rho''(A'')=\C1$.

\medskip

\noindent Assume now that $K\geq2$. Put $\rho_\kappa=(\chi_\kappa\ot\id)\rho$,
and let $\Ucal_\kappa$ be the C*-algebra generated by the coefficients of
$\rho_\kappa$. By \cite{FMP24},
$\Ucal=*_{\kappa=1}^K\Ucal_\kappa$, $\widetilde\omega=*_{\kappa=1}^K\widetilde\omega_\kappa$ and $\Ucal''=*_{\kappa=1}^K(\Ucal_\kappa'',\widetilde\omega_\kappa)$.

\medskip

\noindent Fix $1\leq\kappa\leq K$, and set
$D_\kappa=M_{N_\kappa}(\C)\ot1$ and
$P_\kappa=M_{N_\kappa}(\C)\ot\Ucal_\kappa''$, and put
$R_\kappa=M_{N_\kappa}(\C)\ot
\left(*_{\lambda\neq\kappa}\Ucal_\lambda''\right)$.
With the canonical conditional expectations onto $D_\kappa$, one has
$$M_{N_\kappa}(\C)\ot\Ucal''=P_\kappa*_{D_\kappa}R_\kappa.$$
Moreover, $Q_\kappa:=\rho_\kappa''(A'')\subset P_\kappa$ is diffuse and
is the range of a faithful normal conditional expectation, since
$P_\kappa=M_{N_\kappa}(\C)*A''$
with $Q_\kappa$ identified with the free component $A''$. Applying the
localization result above to this amalgamated free product, and then the
case $K=1$ to $P_\kappa$, gives
$
Q_\kappa'\cap(M_{N_\kappa}(\C)\ot\Ucal'')
\subset Q_\kappa'\cap P_\kappa
\subset Q_\kappa.
$
Since $A''$ is a factor, this implies
$$Q_\kappa'\cap(M_{N_\kappa}(\C)\ot\Ucal'')
\subset\mathcal Z(Q_\kappa)=\C1.$$

\noindent Let $x\in\rho''(A'')'\cap(B\ot\Ucal'')$. For every
$1\leq\kappa\leq K$, the preceding inclusion yields
$(\chi_\kappa\ot\id)(x)\in\C1_{N_\kappa}\ot1$.
Consequently,
$x\in\bigoplus_{\kappa=1}^K\C1_{N_\kappa}\ot1
=\mathcal Z(B)\ot1$.

\end{proof}

\noindent We will use an explicit description, different from the one given in \cite{FMP24}, of the universal quantum homomorphism from $A$ to $B$. This description is given in the next Proposition.

\begin{proposition}\label{PropUQH}
The universal quantum homomorphism $\rho\,:\, A\rightarrow B\ot \Ucal$ is given by
$$\Ucal=\underset{1\leq \kappa\leq K}{*}e^\kappa_{11}(A*M_{N_\kappa}(\C))e^\kappa_{11}\text{ and } \rho(a)=\sum_{\kappa,i,j}e^{\kappa}_{ij}\ot \nu_\kappa(e^\kappa_{1i}ae^\kappa_{j1}),$$
where $\nu_\kappa\,:\,e^\kappa_{11}(A*M_{N_\kappa}(\C))e^\kappa_{11}\rightarrow \Ucal$ is the canonical inclusion. Under this identification, the state $\widetilde{\omega}$ on $\Ucal$ is given by the free product state $\underset{1\leq k\leq K}{*}\varphi_k^{e^k_{11}}$, where $\varphi_k:=\omega*{\rm tr}_k\in (A* M_{N_k}(\C))^*$ and ${\rm tr}_k$ is the normalized trace on $M_{N_k}(\C)$.
\end{proposition}

\medskip

\noindent Proposition~\ref{PropUQH} differs from
\cite[Proposition~2.5]{FMP24} in two precise respects.  The latter realizes
each free-product component as the relative commutant
$M_{N_\kappa}(\C)'\cap(A*M_{N_\kappa}(\C))$ and writes the coefficients as
$\sum_r e^\kappa_{ri}ae^\kappa_{jr}$.  Here the same component is realized as
the corner $e^\kappa_{11}(A*M_{N_\kappa}(\C))e^\kappa_{11}$, with coefficients
$e^\kappa_{1i}ae^\kappa_{j1}$.  In addition, Proposition~\ref{PropUQH}
identifies the canonical state under this corner realization as the free
product of the compressed states; this state identification is not part of
\cite[Proposition~2.5]{FMP24}.

\begin{proof}

\medskip

\noindent By \cite[Lemma~2.4(1)]{FMP24}, it suffices to prove the assertions for
$K=1$. Put $B=M_N(\C)$ and $P=e_{11}(A*M_N(\C))e_{11}$.
For $a,b\in A$, one has
$$
\left(\sum_{i,j=1}^Ne_{ij}\ot e_{1i}ae_{j1}\right)^*
=\sum_{i,j=1}^Ne_{ij}\ot e_{1i}a^*e_{j1},
$$
$$
\left(\sum_{i,j=1}^Ne_{ij}\ot e_{1i}ae_{j1}\right)
\left(\sum_{k,l=1}^Ne_{kl}\ot e_{1k}be_{l1}\right)
=\sum_{i,l=1}^Ne_{il}\ot e_{1i}abe_{l1},
$$
and
$$
\sum_{i,j=1}^Ne_{ij}\ot e_{1i}e_{j1}
=1_{M_N(\C)}\ot e_{11}.
$$
Hence $\rho(a)=\sum_{i,j=1}^Ne_{ij}\ot e_{1i}ae_{j1}$
defines a unital $*$-homomorphism from $A$ to $M_N(\C)\ot P$.

\medskip

\noindent Let $C$ be a unital C*-algebra and let
$\rho':A\to M_N(\C)\ot C$ be a unital $*$-homomorphism. The universal
property of the full free product gives a unique unital $*$-homomorphism
$$\pi_0:A*M_N(\C)\to M_N(\C)\ot C\text{ such that }\pi_0|_A=\rho'\text{ and }\pi_0(b)=b\ot1\text{ for all }b\in M_N(\C).$$
Since $\pi_0(P)\subset(e_{11}\ot1)(M_N(\C)\ot C)(e_{11}\ot1)
=\C e_{11}\ot C$, there exists a unique unital $*$-homomorphism $\pi:P\to C$ satisfying
$\pi_0(y)=e_{11}\ot\pi(y)$, $y\in P$. Writing
$$\rho'(a)=\sum_{i,j=1}^Ne_{ij}\ot\rho'_{ij}(a),$$
one obtains $e_{11}\ot\pi(e_{1i}ae_{j1})
=\pi_0(e_{1i}ae_{j1})
=(e_{1i}\ot1)\rho'(a)(e_{j1}\ot1)
=e_{11}\ot\rho'_{ij}(a)$. Therefore, $(\id\ot\pi)\rho(a)
=\sum_{i,j=1}^Ne_{ij}\ot\rho'_{ij}(a)=\rho'(a)$.
Since $P$ is generated by the elements $e_{1i}ae_{j1}$, the morphism
$\pi$ is unique. Thus $(\rho,P)$ satisfies the universal property.

\medskip

\noindent It remains to identify the state. Put
$M=A*M_N(\C)$, $Q=M_N(\C)'\cap M$, and consider the state 
$\varphi=\omega*{\rm tr}$.
By \cite[Proposition~2.5]{FMP24}, the universal algebra is identified
with $Q$, while \cite[Section~3.1.3]{FMP24} gives $\widetilde\omega=\varphi|_Q$. Define
$\theta:Q\to P$, $\theta(x)=e_{11}xe_{11}$.
The map $\theta$ is a unital $*$-isomorphism with inverse
$\eta:P\to Q$, $\eta(y)=\sum_{r=1}^Ne_{r1}ye_{1r}$.
Indeed, for all $1\leq i,j\leq N$ and $y\in P$,
$e_{ij}\eta(y)=e_{i1}ye_{1j}=\eta(y)e_{ij}$,
and direct computations give $\theta\eta=\id_P$ and
$\eta\theta=\id_Q$.

\medskip

\noindent Let $M_N(\C)^\circ=\ker({\rm tr})$ and $A^\circ=\ker(\omega)$.
The algebraic free product is linearly spanned by $1$ and by the reduced
words in $A^\circ$ and $M_N(\C)^\circ$. If $b\in M_N(\C)^\circ$ and $z$
is a reduced word, the defining formula for the free product state gives
$\varphi(bz)=\varphi(zb)=0$,
unless $z\in M_N(\C)^\circ$, in which case
$\varphi(bz)={\rm tr}(bz)={\rm tr}(zb)=\varphi(zb)$.
By linearity and continuity, $\varphi(by)=\varphi(yb)$ for
$b\in M_N(\C)$ and $y\in M$.
Thus $M_N(\C)\subset M_\varphi$.
For every $x\in Q$, the centralizer identity and the relation
$x\in M_N(\C)'$ give
$$
\varphi(x)
=\sum_{r=1}^N\varphi(e_{rr}x)
=\sum_{r=1}^N\varphi(e_{r1}e_{1r}x)
=\sum_{r=1}^N\varphi(e_{1r}xe_{r1})
=N\varphi(e_{11}xe_{11}).
$$
Since $\varphi(e_{11})=N^{-1}$, it follows that
$\varphi^{e_{11}}(\theta(x))
=N\varphi(e_{11}xe_{11})
=\varphi(x)=\widetilde\omega(x)$.
Thus $\theta$ carries $\widetilde\omega$ to $\varphi^{e_{11}}$. The
general case follows from the free product decompositions in
\cite[Lemma~2.4(1)]{FMP24} and from the definition of
$\widetilde\omega$ as the free product of the component states.

\end{proof}

\subsection{An explicit formula for the canonical state}
\label{SectionUnivGroups}

In \cite{FMP24}, a canonical state $\widetilde{\omega}$
on $\Ucal$ is constructed without an explicit formula.
We establish such a formula for the universal quantum
homomorphism $\rho:C(\mathbb G)\to C(K)\ot\Ucal$,
where $\mathbb G$ is a CQG with Haar states  $h_{\mathbb G}$ and $K$ is a finite quantum
group with  haar states $h_K$.

Fix a Haar-orthonormal basis $\mathcal B_K$ of $C(K)$,
and write
\[
\rho(a)=\sum_{x\in\mathcal B_K}x\ot\nu_x(a),
\qquad
\nu_x(a)=(h_K(x^*\,\cdot)\ot\id)\rho(a).
\]
For $\vec x=(x_1,\ldots,x_n)\in\mathcal B_K^n$,
set $\nu_{\vec x}(\vec a)=\prod_{i=1}^n\nu_{x_i}(a_i)$
and, for $\pi\in\NC(n)$, define
\[
\Theta_{\vec x}(\pi)=
\prod_{\substack{V=\{i_1<\cdots<i_r\}\\V\in\pi}}
h_K\bigl((x_{i_1}\cdots x_{i_r})^*\bigr).
\]

\begin{theorem}\label{ThmUnivState}
The state $\widetilde{\omega}$ on $\Ucal$ is the
unique state satisfying
\[
\widetilde{\omega}(\nu_{\vec x}(\vec a))
=
\sum_{\pi\in\NC(n)}
\frac{\kappa_{\pi}^{h_{\mathbb G}}(\vec a)}
{D^{n-|\pi|}}\Theta_{\vec x}(\pi),
\]
for all $n\geq1$, $\vec x\in\mathcal B_K^n$ and
$\vec a\in C(\mathbb G)^n$, where $D=\dim C(K)$.
Moreover, if $\mathbb G$ is of Kac type,
then $\widetilde{\omega}$ is a trace.
\end{theorem}

\begin{proof}

The uniqueness being obvious, let us show the existence. Let $\Irr(\widehat K)$ be a complete set of representatives of the equivalence classes of irreducible unitary representations of $\widehat K$ and write $C(K)=\bigoplus_{u\in\Irr(\widehat K)}M_{N_u}(\C)$. Let $(e^u_{ij})$ be matrix units in $M_{N_u}(\C)$. By Proposition \ref{PropUQH}, the universal quantum homomorphism $\rho\,:\, C(\mathbb G)\rightarrow C(K)\ot \Ucal$ is given by
$$
\Ucal=\underset{u\in\Irr(\widehat K)}{*}e^u_{11}(C(\mathbb G)* M_{N_u}(\C))e^u_{11}\text{ and }\rho(a)=\sum_{u\in\Irr(\widehat K)}\sum_{i,j=1}^{N_u}e^u_{ij}\ot e^u_{1i}a^ue^u_{j1},
$$
where $a^u$ is the element $a$ viewed in $C(\mathbb G)\subset C(\mathbb G)* M_{N_u}(\C)\subset \underset{u\in\Irr(\widehat K)}{*} C(\mathbb G)*M_{N_u}(\C)$.

\noindent Moreover, $\widetilde{\omega}\in \Ucal^*$ is defined by the free product state $\underset{u\in\Irr(\widehat K)}{*}\varphi^{p_u}_u$, where $\varphi_u:=h_{\mathbb G}*{\rm tr}_u\in (C(\mathbb G)* M_{N_u}(\C))^*$ is the free product of the Haar state $h_{\mathbb G}$ on $C(\mathbb G)$ and the normalized trace on $M_{N_u}(\C)$ and $p_u:=e^u_{11}$ is viewed in $C(\mathbb G)*M_{N_u}(\C)$.

\noindent Let us show that $\widetilde{\omega}$ satisfies the property of the theorem. We will need the following claim.

\medskip

\noindent\textbf{Claim.}\textit{
For every $u\in\Irr(\widehat K)$,
fix the corresponding irreducible $*$-representation
$\pi_u:C(K)\to B(H_u)$ and an orthonormal
basis $(e_i^u)_{i=1}^{N_u}$ of $H_u$.}
Let
\[
\rho_{ij}^u(a):=e^u_{1i}a^ue^u_{j1},
\qquad
c_{ij}^u(x):=\frac{N_u}{\dim C(K)}
\langle e_i^u,\pi_u(x)e_j^u\rangle.
\]
Then, for every $x\in\mathcal B_K$, $\nu_x(a)=
\sum_{u\in\Irr(\widehat K)}
\sum_{i,j=1}^{N_u}
c_{ij}^u(x)\rho_{ij}^u(a).$

\noindent\textit{Proof of the Claim.}
By the Haar trace formula
\cite[Theorem~2.2 ]{vandaele97},
we have
\[
\begin{aligned}
\nu_x(a)
&=(h_K(x^*\,\cdot)\ot\id)\rho(a)\\
&=\sum_{u,i,j}h_K(x^*e^u_{ij})\rho^u_{ij}(a)\\
&=\sum_{u,i,j}\frac{N_u}{\dim C(K)}
\langle e_i^u,\pi_u(x)e_j^u\rangle
\rho^u_{ij}(a),
\end{aligned}
\]
which proves the claim. \hfill$\qed$

\medskip

\noindent\textit{End of the proof of Theorem \ref{ThmUnivState}.}
By the moment-cumulant formula (Theorem \ref{ThmCumulant}), we have
\[
\widetilde{\omega}(\nu_{\vec{x}}(\vec{a}))
=\sum_{\pi\in\NC(n)}
\kappa_\pi^{\widetilde{\omega}}
[\nu_{x_1}(a_1),\dots,\nu_{x_n}(a_n)].
\]
We first compute the free cumulants appearing in this formula.

For $n\geq 1$, $\vec{u}\in\Irr(\widehat K)^n$ we write
$N_{\vec{u}}:=\{1,\dots,N_{u_1}\}\times\dots\times
\{1,\dots,N_{u_n}\}$.
For $\vec{a}\in C(\mathbb G)^n$,
$\vec{i}=(i_1,\dots,i_n)\in N_{\vec{u}}$ and
$\vec{j}=(j_1,\dots,j_n)\in N_{\vec{u}}$ define
$\rho_{\vec{i},\vec{j}}^{\vec{u}}(\vec{a})
:=\rho^{u_1}_{i_1j_1}(a_1)\dots
\rho^{u_n}_{i_nj_n}(a_n)\in\Ucal$
and, for $\vec{x}\in\mathcal B_K^n$, define
$C_{\vec{i},\vec{j}}^{\vec{u}}(\vec{x})
:=c^{u_1}_{i_1j_1}(x_1)\dots
c^{u_n}_{i_nj_n}(x_n)\in\C$. By the Claim, we have
\[
\nu_{\vec{x}}(\vec{a})
=\sum_{\vec{u}}\sum_{\vec{i},\vec{j}\in N_{\vec{u}}}
C_{\vec{i},\vec{j}}^{\vec{u}}(\vec{x})
\rho_{\vec{i},\vec{j}}^{\vec{u}}(\vec{a}).
\]
By the multilinearity of free cumulants, we have:
$$\kappa_n^{\widetilde{\omega}}
[\nu_{x_1}(a_1),\dots,\nu_{x_n}(a_n)]=\sum_{\vec{u}}\sum_{\vec{i},\vec{j}\in N_{\vec{u}}}
C_{\vec{i},\vec{j}}^{\vec{u}}(\vec{x})
\kappa_n^{\widetilde{\omega}}
[\rho^{u_1}_{i_1j_1}(a_1),\dots,
\rho^{u_n}_{i_nj_n}(a_n)].$$
By Theorem \ref{VanishingCumulant}, one has
$\kappa_n^{\widetilde{\omega}}
[\rho^{u_1}_{i_1j_1}(a_1),\dots,
\rho^{u_n}_{i_nj_n}(a_n)]=0$
whenever $s\mapsto u_s$ is not constant.
Hence, we get:

$$\kappa_n^{\widetilde{\omega}}
[\nu_{x_1}(a_1),\dots,\nu_{x_n}(a_n)]=\sum_{u\in\Irr(\widehat K)}
\sum_{\vec{i},\vec{j}\in\{1,\dots,N_u\}^n}
C_{\vec{i},\vec{j}}^{(u,\dots,u)}(\vec{x})
\kappa_n^{\widetilde{\omega}}
[\rho^u_{i_1j_1}(a_1),\dots,
\rho^u_{i_nj_n}(a_n)].$$

Fix $u\in\Irr(\widehat K)$.
Since each $\rho^u_{i_sj_s}(a_s)$ is in $\Ucal_u:=e^u_{11}(C(\mathbb G)*M_{N_u}(\C))e^u_{11}$
and since $\widetilde{\omega}\vert_{\Ucal_u}
=\varphi_u^{p_u}$, we deduce that:

$$X:=\kappa_n^{\widetilde{\omega}}
[\rho^u_{i_1j_1}(a_1),\dots,
\rho^u_{i_nj_n}(a_n)]\\=\kappa_n^{\varphi_u^{p_u}}
[\rho^u_{i_1j_1}(a_1),\dots,
\rho^u_{i_nj_n}(a_n)].$$

\noindent Using Theorem \ref{Thm14.18}, we find:
\[
X=
\left\{
\begin{aligned}
&\frac{1}{N_u^{n-1}}
\kappa_n^{\varphi_u}[a^u_1,\dots,a^u_n]
&&\text{if $(\vec{i},\vec{j})$ is cyclic,}\\
&0 &&\text{otherwise.}
\end{aligned}
\right.
\]
Since $a^u_1,\dots,a^u_n\in C(\mathbb G)
\subset C(\mathbb G)*M_{N_u}(\C)$
and since $\varphi_u\vert_{C(\mathbb G)}
=h_{\mathbb G}$, we deduce that
$\kappa_n^{\varphi_u}[a^u_1,\dots,a^u_n]
=\kappa_n^{h_{\mathbb G}}[a_1,\dots,a_n]$.
Hence,
\[
X=
\left\{
\begin{aligned}
&\frac{1}{N_u^{n-1}}
\kappa_n^{h_{\mathbb G}}[a_1,\dots,a_n]
&&\text{if $(\vec{i},\vec{j})$ is cyclic,}\\
&0 &&\text{otherwise.}
\end{aligned}
\right.
\]

It follows that:
\[
\begin{aligned}
&\kappa_n^{\widetilde{\omega}}
[\nu_{x_1}(a_1),\dots,\nu_{x_n}(a_n)]\\
&=\kappa_n^{h_{\mathbb G}}(\vec a)
\sum_{u\in\Irr(\widehat K)}
\frac{1}{N_u^{n-1}}
\sum_{i_1,\dots,i_n=1}^{N_u}
c^u_{i_1i_2}(x_1)
c^u_{i_2i_3}(x_2)\cdots
c^u_{i_ni_1}(x_n).
\end{aligned}
\]
The cyclicity condition
$j_1=i_2,\dots,j_{n-1}=i_n,j_n=i_1$
determines all the indices $j_1,\dots,j_n$.
Recall that
$c^u_{ij}(x)=\frac{N_u}{D}
\langle e_i^u,\pi_u(x)e_j^u\rangle$,
where $D=\dim C(K)$. We have:
\[
\begin{aligned}
&\sum_{u\in\Irr(\widehat K)}N_u^{1-n}
\sum_{i_1,\dots,i_n=1}^{N_u}
c^u_{i_1i_2}(x_1)\dots c^u_{i_ni_1}(x_n)\\
&=\frac{1}{D^n}
\sum_{u\in\Irr(\widehat K)}N_u
\sum_{i_1,\dots,i_n=1}^{N_u}
\prod_{s=1}^n
\langle e^u_{i_s},\pi_u(x_s)e^u_{i_{s+1}}\rangle\\
&=\frac{1}{D^n}
\sum_{u\in\Irr(\widehat K)}N_u
\overline{\operatorname{Tr}
(\pi_u(x_1\cdots x_n))}\\
&=\frac{1}{D^{n-1}}
h_K((x_1\cdots x_n)^*),
\end{aligned}
\]
where $i_{n+1}=i_1$ and we used the Haar trace
formula in the last equality.

Hence,
\[
\kappa_n^{\widetilde{\omega}}
[\nu_{x_1}(a_1),\dots,\nu_{x_n}(a_n)]
=
\frac{\kappa_n^{h_{\mathbb G}}(\vec a)}
{D^{n-1}}h_K((x_1\cdots x_n)^*).
\]

\noindent In conclusion, by Theorem \ref{ThmCumulant},
we have:
\[
\begin{aligned}
\widetilde{\omega}(\nu_{\vec{x}}(\vec{a}))
&=\sum_{\pi\in\NC(n)}
\prod_{\substack{V=\{v_1<\dots<v_r\}\\V\in\pi}}
\kappa_r^{\widetilde{\omega}}
[\nu_{x_{v_1}}(a_{v_1}),\dots,
\nu_{x_{v_r}}(a_{v_r})]\\
&=\sum_{\pi\in\NC(n)}
\frac{\kappa_\pi^{h_{\mathbb G}}(\vec a)}
{D^{n-|\pi|}}\Theta_{\vec{x}}(\pi).
\end{aligned}
\]

\medskip

\end{proof}

\section{Free wreath product and crossed product}

\subsection{Free wreath products by finite quantum groups}

\begin{definition}
Let $\mathbb G$ be a compact quantum group and $K$ a finite quantum
group, with $\beta_\Delta=\Delta_K$. Consider the free wreath product
$\mathbb W=\mathbb G\wr_{*,\beta_\Delta}K$ of \cite{FT25}.
Fix a Haar-orthonormal basis
$B_K=\{z_1=1,\ldots,z_{\dim C(K)}\}$ of $C(K)$.

Let $\rho:C(\mathbb G)\to C(K)\ot\Ucal$ be the universal
quantum homomorphism.
For $x\in C(K)$ and $a\in C(\mathbb G)$, set
$$
\nu_x(a)=(h_K(x^*\,\cdot\,)\ot\id)\rho,
\qquad
\rho(a)=\sum_{z\in B_K}z\ot\nu_z(a).
$$
By the generalized free wreath product construction of
Fima and Troupel \cite[Section~3]{FT25},
$C(\mathbb W)$ is the universal unital C*-algebra
generated by $C(K)$ and the coefficients
$\nu_z(a)$, $z\in\mathcal B_K$, $a\in C(\mathbb G)$,
of the universal quantum homomorphism $\rho$,
subject to the defining relations of \cite{FT25}.
\begingroup
\setlength{\jot}{2pt}
\begin{align}
\nu_x(\lambda a+\mu b)
&=\lambda\nu_x(a)+\mu\nu_x(b),
\quad \nu_x(1)=h_K(x^*)1, \notag\\
\nu_x(a)^*
&=\nu_{x^*}(a^*), \notag\\
\nu_x(ab)
&=\sum_{y,z\in B_K}h_K(x^*yz)\nu_y(a)\nu_z(b),
\notag\\
k\nu_x(a)
&=\sum_{(k)}
\nu_{S(k_{(1)})^*xk_{(2)}^*}(a)k_{(3)}.
\label{eq:finite-quantum-generators}
\end{align}
\endgroup
Here $a,b\in C(\mathbb G)$, $x,k\in C(K)$ and
$\lambda,\mu\in\mathbb C$.

The coproduct is determined by
$\Delta_{\mathbb W}(k)=\Delta_K(k)$ for $k\in C(K)$, and
\begin{equation}\label{eq:finite-quantum-coproduct}
\Delta_{\mathbb W}(\nu_x(a))
=\sum_{(a)}\sum_{y,z\in B_K}\sum_{(y)}
h_K(x^*zy_{(1)})
\nu_z(a_{(1)})y_{(2)}
\ot\nu_y(a_{(2)})=\sum_{(a)}\sum_{y\in B_K}\sum_{(y)}
\nu_{xy_{(1)}^*}(a_{(1)})\,y_{(2)}
\otimes \nu_y(a_{(2)}).
\end{equation}

\end{definition}

\subsection{Free wreath products by finite group duals}

We now consider the special case $K=\widehat{\Lambda}$,
where $\Lambda$ is a finite group. In this case,
$C(K)=C^*(\Lambda)$ and the canonical group basis
$B_K=\Lambda$ is Haar-orthonormal.

The preceding relations reduce to
\begingroup
\setlength{\jot}{2pt}
\begin{align*}
\nu_\gamma(\lambda a+\mu b)
&=\lambda\nu_\gamma(a)+\mu\nu_\gamma(b),\\
\nu_\gamma(a)^*
&=\nu_{\gamma^{-1}}(a^*),
\quad \nu_\gamma(1)=\delta_{\gamma,e}1,\\
\nu_\gamma(ab)
&=\sum_{rs=\gamma}\nu_r(a)\nu_s(b),\\
s\nu_{rs}(a)&=\nu_{sr}(a)s,
\end{align*}
\endgroup
for $a,b\in C(\mathbb G)$,
$\gamma,r,s\in\Lambda$ and $\lambda,\mu\in\mathbb C$.

The coproduct becomes
$$
\Delta_{\mathbb W}(\gamma)=\gamma\ot\gamma,
\qquad
\Delta_{\mathbb W}(\nu_\gamma(a))
=\sum_{rs=\gamma}
\nu_r(a_{(1)})s\ot\nu_s(a_{(2)}).
$$

For finite group duals, the free wreath product also admits
a description in terms of quantum automorphism groups of finite
quantum graphs in the sense of \cite{FP16}.
In particular, for a non-abelian finite group $\Lambda$,
the quantum Cayley graph construction of \cite{BGMSW25}
provides a finite quantum graph
$(C^*(\Lambda),\tau_\Lambda,d_\Lambda)$ such that
$$
\operatorname{Aut}^+
(C^*(\Lambda),\tau_\Lambda,d_\Lambda)
\simeq\widehat{\Lambda}.
$$
Consequently, the corresponding free wreath product can be
identified with the Fima--Pittau free wreath product
$$
\mathbb G\wr_{*,\beta_\Delta}\widehat{\Lambda}
\simeq
\mathbb G\wr_*
\operatorname{Aut}^+
(C^*(\Lambda),\tau_\Lambda,d_\Lambda).
$$
The proof of this identification is deferred to the Appendix.

\begin{lemma}\label{induced-action}
Let $\mathbb G$ be an arbitrary compact quantum group
and $K$ a finite quantum group. Then there exists a canonical right action $\delta_{\Ucal}:\Ucal\to\Ucal\ot C(\widehat K)$
induced by the right adjoint coaction of $\widehat K$ on $C(K)$.

Moreover, the corresponding $C(K)$-action on $\Ucal$ satisfies $$k\triangleright\nu_x(a)
=\sum_{(k)}
\nu_{S(k_{(1)})^*xk_{(2)}^*}(a),
\qquad k,x\in C(K).$$
\end{lemma}

The induced action on $\Ucal$.

\begin{proof}
Let $\rho:C(\mathbb G)\to C(K)\ot\Ucal$ be the universal quantum
homomorphism and $\delta_{\mathrm{ad}}$ be the adjoint coaction, characterized by
$$(\operatorname{ev}_{k}\ot\id)\delta_{\mathrm{ad}}(z)
 =\sum_{(k)}S(k_{(1)})z k_{(2)},\qquad k,z\in C(K),$$

By the universal property of $\Ucal$, it induces
$\delta_{\Ucal}:\Ucal\to\Ucal\ot C(\widehat K)$, specified by
\[
\begin{tikzcd}[column sep=4em,row sep=2.8em]
 C(\mathbb G)\arrow[rr,"\rho"]\arrow[d,"\rho"']
 && C(K)\ot\Ucal\arrow[d,"\id\ot\delta_{\Ucal}"]\\
 C(K)\ot\Ucal\arrow[r,"\delta_{\mathrm{ad}}\ot\id"']
 & C(\widehat K)\ot C(K)\ot\Ucal\arrow[r,"\Sigma"']
 & C(K)\ot\Ucal\ot C(\widehat K).
\end{tikzcd}
\]
where $\Sigma$ permutes the $C(\widehat K)$ with $C(K)\ot\Ucal$. 

The above diagram
determines a unique unital $*$-homomorphism $\delta_{\Ucal}:\Ucal\to\Ucal\ot C(\widehat K)$. Since $\delta_{\mathrm{ad}}$ is a action of finite quantum group $\widehat{K}$,
its coassociativity and counit properties,
together with the universal property of $\Ucal$,
imply that

\[
(\delta_{\Ucal}\ot\id)\delta_{\Ucal}
=(\id\ot\Delta_{\widehat K})\delta_{\Ucal},
\qquad
(\id\ot\varepsilon_{\widehat K})\delta_{\Ucal}=\id.
\]
Moreover, since $\widehat K$ is finite,
$C(\widehat K)=\operatorname{Pol}(\widehat K)$.
The antipode identities, coassociativity
and the counit imply the Podleś condition. Thus $\delta_{\Ucal}$ defines a continuous right
coaction of $\widehat K$ on $\Ucal$.

We have
\[
(\delta_{\mathrm{ad}}\ot\id)\rho(a)
 =\sum_{z,k\in B_K}h_K(k^*\,\cdot\,)\ot
 \Bigl(\sum_{(k)}S(k_{(1)})z k_{(2)}\Bigr)\ot\nu_z(a).
\]
For $k'\in C(K)$, the induced action on the coefficients is
\begin{align*}
k'\triangleright\nu_x(a)
&=(\id\ot\operatorname{ev}_{k'})\delta_{\Ucal}
   \bigl[(h_K(x^*\,\cdot\,)\ot\id)\rho(a)\bigr]\\
&=(h_K(x^*\,\cdot\,)\ot\id)
   (\operatorname{ev}_{k'}\ot\id\ot\id)
   (\delta_{\mathrm{ad}}\ot\id)\rho(a)\\
&=\sum_{z\in B_K}h_K\!\left(
 x^*\sum_{(k')}S(k'_{(1)})zk'_{(2)}\right)\nu_z(a)\\
&=\sum_{z\in B_K}h_K\!\left(
 \sum_{(k')}k'_{(2)}x^*S(k'_{(1)})z\right)\nu_z(a)\\
&=\sum_{(k')}\nu_{S(k'_{(1)})^*x{k'}_{(2)}^*}(a).
\end{align*}
Here we used the traciality of $h_K$, since a finite quantum group is of
Kac type.
\end{proof}

\begin{proposition}\label{prop:FT-smash}
There is a canonical C*-algebra isomorphism
\[
C(\mathbb G\wr_{*,\beta_\Delta}K)
\cong \Ucal\#_{\mathrm u}C(K),
\]
where $\Ucal\#_{\mathrm u}C(K)$ denotes the universal
C*-smash product associated with the action of
$C(K)$ on $\Ucal$. It is generated by the
canonical copies of $\Ucal$ and $C(K)$, subject
to the covariance relations
\begin{equation}\label{eq:smash-finite-quantum}
k\nu_x(a)=\sum_{(k)}
\nu_{S(k_{(1)})^*xk_{(2)}^*}(a)\,k_{(3)},
\quad k,x\in C(K),\ a\in C(\mathbb G).
\end{equation}
\end{proposition}
\begin{proof}
Let $\rho_W:C(\mathbb G)\to C(K)\ot C(\mathbb W)$ denote the
universal quantum homomorphism with coefficients in $C(\mathbb W)$.
The FT relation in this case is
\begin{equation}\label{eq:FT-finite-quantum}
 \rho_W(a)\Delta_K(k)=\Delta_K(k)\rho_W(a).
\end{equation}

Indeed, the relation \eqref{eq:smash-finite-quantum} is equivalent to
\begin{equation}\label{eq:smash-ft-intermediate}
 (1\ot k)\rho(a)
 =\sum_{(k)}(S(k_{(1)})\ot1)\rho(a)
 (k_{(2)}\ot k_{(3)}).
\end{equation}
 for every $x\in C(K)$ its left-hand Haar slice is $k\nu_x(a)$,
whereas the right-hand slice is
\begin{align*}
&\sum_{(k)}\nu_{S(k_{(1)})^*xk_{(2)}^*}(a)k_{(3)}\\
&\quad=\sum_{(k)}
\bigl[h_K\bigl((S(k_{(1)})^*xk_{(2)}^*)^*\,\cdot\,\bigr)
\ot\id\bigr]\rho(a)(1\ot k_{(3)})\\
&\quad=(h_K(x^*\,\cdot\,)\ot\id)
\sum_{(k)}(S(k_{(1)})\ot1)\rho(a)
 (k_{(2)}\ot k_{(3)}),
\end{align*}
using traciality of $h_K$. Since the Haar slices span $C(K)^*$,
\eqref{eq:smash-ft-intermediate} follows.

Now, \eqref{eq:smash-ft-intermediate} implies the FT relation:
\begin{align*}
\Delta_K(k)\rho(a)
&=\sum_{(k)}(k_{(1)}\ot1)(1\ot k_{(2)})\rho(a)\\
&=\sum_{(k)}(k_{(1)}S(k_{(2)})\ot1)
   \rho(a)(k_{(3)}\ot k_{(4)})\\
&=\sum_{(k)}\varepsilon(k_{(1)})\rho(a)
   (k_{(2)}\ot k_{(3)})
 =\rho(a)\Delta_K(k).
\end{align*}
For the reverse implication \emph{FT $\Rightarrow$ smash product},
apply \eqref{eq:FT-finite-quantum} to $k_{(2)}$ in Sweedler notation:
\begin{align*}
&\sum_{(k)}(S(k_{(1)})\ot1)\rho(a)
  (k_{(2)}\ot k_{(3)})\\
&\quad=\sum_{(k)}(S(k_{(1)})\ot1)\rho(a)\Delta_K(k_{(2)})\\
&\quad=\sum_{(k)}(S(k_{(1)})\ot1)\Delta_K(k_{(2)})\rho(a)\\
&\quad=\sum_{(k)}(S(k_{(1)})k_{(2)}\ot k_{(3)})\rho(a)
 =(1\ot k)\rho(a).
\end{align*}
This is \eqref{eq:smash-ft-intermediate}, and taking all Haar slices
gives \eqref{eq:smash-finite-quantum}.

Finally, $C(\mathbb W)$ is universal for the coefficients of $\rho_W$
and $C(K)$ subject to the FT relations, whereas the full smash product is
universal for the coefficients of $\rho$ and $C(K)$ subject to
\eqref{eq:smash-finite-quantum}. Their equivalence yields mutually inverse
$*$-homomorphisms fixing both sets of generators, hence the asserted
isomorphism.
\end{proof}

\noindent In the special case where $K=\widehat{\Lambda}$ for a finite group
$\Lambda$, the smash-product decomposition obtained in
Proposition~\ref{prop:finite-quantum-smash} reduces to a crossed-product
decomposition. We now describe the corresponding action of $\Lambda$
on $\Ucal$.

In this case, $\Ucal$ is the universal unital C*-algebra generated
by the coefficients of a unital $*$-homomorphism
$\rho:C(\mathbb G)\rightarrow C^*(\Lambda)\ot\Ucal$.
For each $r\in\Lambda$, the universal property of $\Ucal$, applied
to the unital $*$-homomorphism
$\rho^{(r)}:C(\mathbb G)\rightarrow C^*(\Lambda)\ot\Ucal$
defined by
$\rho^{(r)}=({\rm Ad}(r)\ot\id)\rho$,
where ${\rm Ad}(r):C^*(\Lambda)\rightarrow C^*(\Lambda)$ is given by
${\rm Ad}(r)(s)=r^{-1}sr$,
gives a unique unital $*$-homomorphism
$\alpha_r:\Ucal\rightarrow\Ucal$ such that
$$
(\id\ot\alpha_r)\rho=\rho^{(r)}
\quad\text{for all }r\in\Lambda.
$$
In the sequel, we write, for all $a\in C(\mathbb G)$,
$\rho(a)=\sum_{\gamma\in\Lambda}\gamma\ot\rho_\gamma(a),$ where $\rho_\gamma(a)\in\Ucal$.

\begin{corollary}\label{PropCrossed}
The map $r\mapsto\alpha_r$ is an action of $\Lambda$ on $\Ucal$ by automorphisms and there is a unique isomorphism $\psi\,:\, \Ucal\rtimes_\alpha\Lambda\rightarrow C(\mathbb W)$ such that $$\psi(r)=r\text{ for all }r\in \Lambda\text{ and }\psi(\rho_\gamma(a))=\nu_\gamma(a)\text{ for all }\gamma\in\Lambda,\,a\in C(\mathbb G),$$
where $\Ucal\rtimes_\alpha\Lambda$ denotes the full crossed product, generated by $\Ucal$ and $C^*(\Lambda)$ with the relations $rar^{-1}=\alpha_r(a)$ for all $a\in \Ucal$ and $r\in\Lambda$.
\end{corollary}

\begin{proof}
\begingroup
\setlength{\abovedisplayskip}{3pt}
\setlength{\belowdisplayskip}{3pt}
\setlength{\abovedisplayshortskip}{1pt}
\setlength{\belowdisplayshortskip}{2pt}
\setlength{\jot}{1pt}
\setlength{\parskip}{0pt}

\noindent For $r,s\in\Lambda$, one has
${\rm Ad}(rs)={\rm Ad}(s)\circ{\rm Ad}(r)$.
Since $\rho^{(1)}=\rho$, the uniqueness in the universal property of
$\Ucal$ gives $\alpha_1=\id_\Ucal$. Moreover,
$$
\begin{aligned}
(\id\ot(\alpha_r\circ\alpha_s))\rho
&=(\id\ot\alpha_r)({\rm Ad}(s)\ot\id)\rho
=({\rm Ad}(s)\ot\id)(\id\ot\alpha_r)\rho
=({\rm Ad}(s)\circ{\rm Ad}(r)\ot\id)\rho\\
&=(\id\ot\alpha_{rs})\rho.
\end{aligned}
$$
Since the coefficients of $\rho$ generate $\Ucal$,
$\alpha_{rs}=\alpha_r\circ\alpha_s$.
In particular, $\alpha_{r^{-1}}=\alpha_r^{-1}$, and
$r\mapsto\alpha_r$ is an action of $\Lambda$ on $\Ucal$ by
$*$-automorphisms.

\noindent Define
$\nu(a)=\sum_{\gamma\in\Lambda}\gamma\ot\nu_\gamma(a)$, $a\in C(\mathbb G)$.
The first three families of defining relations of $C(\mathbb W)$ give
$\nu(a)^*=\nu(a^*)$, $\nu(a)\nu(b)=\nu(ab)$, and $\nu(1)=1$.
Hence $\nu$ extends uniquely to a unital $*$-homomorphism
$\nu:C(\mathbb G)\to C^*(\Lambda)\ot C(\mathbb W)$.
By the universal property of $\Ucal$, there exists a unique unital
$*$-homomorphism $\psi_0:\Ucal\to C(\mathbb W)$ such that
$(\id\ot\psi_0)\rho=\nu$. Equivalently,
$\psi_0(\rho_\gamma(a))=\nu_\gamma(a)$ for
$\gamma\in\Lambda$ and $a\in C(\mathbb G)$.

\noindent For $r\in\Lambda$ and $a\in C(\mathbb G)$, one has
$$
(\id\ot(\psi_0\circ\alpha_r))\rho(a)
=({\rm Ad}(r)\ot\psi_0)\rho(a)
=({\rm Ad}(r)\ot\id)\nu(a)
=\sum_{\delta\in\Lambda}\delta\ot
\nu_{r\delta r^{-1}}(a).
$$
The fourth family of defining relations of $C(\mathbb W)$ implies
$r\nu_\delta(a)r^{-1}=\nu_{r\delta r^{-1}}(a)$ for
$\delta,r\in\Lambda$. Therefore,
$\psi_0(\alpha_r(a))=r\psi_0(a)r^{-1}$ for
$a\in\Ucal$ and $r\in\Lambda$.
Thus $(\psi_0,r\mapsto r)$ is a covariant representation of
$(\Ucal,\Lambda,\alpha)$ in $C(\mathbb W)$. The universal property of the full
crossed product gives a unique unital $*$-homomorphism
$\psi:\Ucal\rtimes_\alpha\Lambda\to C(\mathbb W)$ such that
$\psi(r)=r$ and $\psi(\rho_\gamma(a))=\nu_\gamma(a)$.

\noindent Conversely, regard $\rho_\gamma(a)$ and $r$ as elements of
$\Ucal\rtimes_\alpha\Lambda$. Since $\rho$ is a unital
$*$-homomorphism, the elements $\rho_\gamma(a)$ satisfy the first three
families of defining relations of $C(\mathbb W)$. Moreover,
$$
s\rho_{rs}(a)s^{-1}
=\alpha_s(\rho_{rs}(a))
=\rho_{s(rs)s^{-1}}(a)
=\rho_{sr}(a),
$$
and hence $s\rho_{rs}(a)=\rho_{sr}(a)s$. The universal property of $C(\mathbb W)$ gives a unique unital
$*$-homomorphism $\chi:C(\mathbb W)\to\Ucal\rtimes_\alpha\Lambda$ such that
$\chi(r)=r$ and $\chi(\nu_\gamma(a))=\rho_\gamma(a)$.
Both $\chi\circ\psi$ and $\psi\circ\chi$ are the identity on the
respective generating sets. Therefore, $\chi=\psi^{-1}$.

\endgroup
\end{proof}
\

\section{The Haar state}

\noindent While \cite{FT25} characterizes the Haar state
through conditional expectations, we obtain a closed
combinatorial formula for all mixed moments of the
canonical generators of $\mathbb W(\mathbb G,K)$.
Remarkably, the formula separates the contributions
of the two quantum groups, involving only the free
cumulants of $h_{\mathbb G}$ and the Haar moments of $K$.

\medskip

\noindent Given $n\geq 1$,
$\vec{x}=(x_1,\dots,x_n)\in\mathcal B_K^n$
and $\vec{a}=(a_1,\dots,a_n)\in\Pol(\mathbb G)^n$,
define
\[
\nu_{\vec{x}}(\vec{a})
:=\nu_{x_1}(a_1)\dots\nu_{x_n}(a_n)
\in C(\mathbb W).
\]
For each $\pi\in\NC(n)$, define
\[
\Theta_{\vec{x}}(\pi)
:=
\prod_{\substack{V=\{r_1<\dots<r_s\}\\V\in\pi}}
h_K\bigl((x_{r_1}\dots x_{r_s})^*\bigr).
\]

\begin{lemma}\label{LemCumulantHaar}
Let $\mathbb G$ be a compact quantum group with Haar state $h$,
and let $\kappa_n^h$ be the associated free cumulants.
Then, for all $a_1,\ldots,a_n\in\Pol(\mathbb G)$,
\[
\begin{aligned}
\sum_{(a_1),\ldots,(a_n)}
\kappa_n^h(a_{1(1)},\ldots,a_{n(1)})
a_{1(2)}\cdots a_{n(2)}
&=\kappa_n^h(a_1,\ldots,a_n)1,\\
\sum_{(a_1),\ldots,(a_n)}
a_{1(1)}\cdots a_{n(1)}
\kappa_n^h(a_{1(2)},\ldots,a_{n(2)})
&=\kappa_n^h(a_1,\ldots,a_n)1.
\end{aligned}
\]
\end{lemma}

\begin{proof}
By M\"obius inversion,
$\kappa_n^h(a_1,\ldots,a_n)
=\sum_{\sigma\in\NC(n)}\mu(\sigma,1_n)
\prod_{V\in\sigma}h(a_V)$,
where $a_V=a_{i_1}\cdots a_{i_r}$ for
$V=\{i_1<\cdots<i_r\}$.
Hence, for the first identity, it suffices to show that
\[
\sum_{(a_1),\ldots,(a_n)}
\left(\prod_{V\in\sigma}h(a_{V(1)})\right)
a_{1(2)}\cdots a_{n(2)}
=\left(\prod_{V\in\sigma}h(a_V)\right)1.
\]
Choose an interval block $V=\{k+1,\ldots,k+r\}$ of $\sigma$.
By Haar invariance,
\[
\sum_{(a_{k+1}),\ldots,(a_{k+r})}
h(a_{k+1(1)}\cdots a_{k+r(1)})
a_{k+1(2)}\cdots a_{k+r(2)}
=h(a_V)1.
\]
Thus $V$ can be removed, and iterating over the remaining
noncrossing partition proves the desired identity.
The second identity is analogous.
\end{proof}

\begin{theorem}\label{ThmHaar}
Let $h_{\mathbb G}$ and $h_K$ be the Haar states of
$\mathbb G$ and $K$, respectively, and set
$D=\dim C(K)$. The Haar state
$h\in C(\mathbb W)^*$ is the unique state such that,
for all $n\geq 1$,
$\vec{x}\in\mathcal B_K^n$,
$\vec a=(a_1,\ldots,a_n)\in\Pol(\mathbb G)^n$,
and all $b\in C(K)$, one has:
\[
h(\nu_{\vec{x}}(\vec a)b)
=
h_K(b)
\sum_{\pi\in\NC(n)}
\frac{\kappa_\pi^{h_{\mathbb G}}(\vec a)}
{D^{n-|\pi|}}
\Theta_{\vec{x}}(\pi).
\]
\end{theorem}

\begin{proof}
The elements $\nu_{\vec{\gamma}}(\vec a)b$,
where $\vec{\gamma}\in\mathcal B_K^n$,
$\vec a\in\Pol(\mathbb G)^n$ and $b\in C(K)$,
linearly span a dense $*$-subalgebra of
$C(\mathbb W)$, proving uniqueness.
Identifying, by the smash-product decomposition,
$C(\mathbb W)$ with $\Ucal\# C(K)$,
it suffices to show that the state
\[
h(ub)=\widetilde{\omega}(u)h_K(b),
\qquad u\in\Ucal,\quad b\in C(K),
\]
is $\Delta$-invariant.
We first consider the elements
$\nu_{\vec{\gamma}}(\vec a)$.

\medskip

\noindent For
$\vec{\gamma}=(\gamma_1,\ldots,\gamma_n)
\in\mathcal B_K^n$ and
$\vec{x}=(x_1,\ldots,x_n)\in\mathcal B_K^n$,
write
\[
\Delta_{\mathbb G}(a_i)
=\sum_{(a_i)}a_{i(1)}\ot a_{i(2)},
\qquad
\vec a_{(j)}
=(a_{1(j)},\ldots,a_{n(j)}),
\quad j=1,2.
\]
By the coproduct formula, we have

$$\Delta(\nu_{\vec{\gamma}}(\vec a))=
\sum_{(a_1),\ldots,(a_n)}
\sum_{\vec{x}\in\mathcal B_K^n}
\sum_{(\vec{x})}
\nu_{\gamma_1x_{1(1)}^*}(a_{1(1)})x_{1(2)}
\cdots
\nu_{\gamma_nx_{n(1)}^*}(a_{n(1)})x_{n(2)}\ot
\nu_{x_1}(a_{1(2)})\cdots
\nu_{x_n}(a_{n(2)}),$$
where $\sum_{(\vec{x})}$ denotes the
iterated Sweedler summation over
$x_1,\ldots,x_n$.

\medskip

\noindent By the smash-product relations,
we move the elements of $C(K)$ to the right
in the first tensor factor.
Define$\vec{\gamma}_{(\vec{x})}
=(d_1,\ldots,d_n),$
where

$$d_1=\gamma_1x_{1(1)}^*,\quad
d_2=S(x_{1(2)}^*)\gamma_2
x_{2(1)}^*x_{1(3)}^*,\quad d_3=S(x_{1(4)}^*)S(x_{2(2)}^*)
\gamma_3x_{3(1)}^*
x_{2(3)}^*x_{1(5)}^*,$$
and, in general,

$$d_k={}
S(x_{1(2k-2)}^*)
S(x_{2(2k-4)}^*)\cdots
S(x_{k-1(2)}^*)\cdot\gamma_k x_{k(1)}^*
x_{k-1(3)}^*x_{k-2(5)}^*
\cdots x_{1(2k-1)}^*.$$
Set
\[
R_n(\vec{x})
:=x_{1(2n)}x_{2(2n-2)}
\cdots x_{n(2)}.
\]
Consequently,

$$\Delta(\nu_{\vec{\gamma}}(\vec a))=
\sum_{(a_1),\ldots,(a_n)}
\sum_{\vec{x}\in\mathcal B_K^n}
\sum_{(\vec{x})}
\nu_{\vec{\gamma}_{(\vec{x})}}
(\vec a_{(1)})
R_n(\vec{x})
\ot\nu_{\vec{x}}(\vec a_{(2)}).$$

\medskip

\noindent It follows from the definition of $h$ and Theorem \ref{ThmUnivState}

$$(h\ot\id)
\Delta(\nu_{\vec{\gamma}}(\vec a))=
\sum_{\pi\in\NC(n)}
\frac{1}{D^{n-|\pi|}}
\sum_{(a_1),\ldots,(a_n)}
\sum_{\vec{x}\in\mathcal B_K^n}
\sum_{(\vec{x})}
\kappa_\pi^{h_{\mathbb G}}
(\vec a_{(1)})\cdot
h_K(R_n(\vec{x}))
\Theta_{\vec{\gamma}_{(\vec{x})}}(\pi)
\nu_{\vec{x}}(\vec a_{(2)}),$$
where $D=\dim C(K)$ and
\[
\Theta_{\vec{\gamma}_{(\vec{x})}}(\pi)
=
\prod_{\substack{
W=\{i_1<\cdots<i_s\}\\W\in\pi}}
h_K\bigl((d_{i_1}\cdots d_{i_s})^*\bigr).
\]
For $\pi\in\NC(n)$, set

$$F_\pi(\vec{\gamma},\vec a)
:={}
\sum_{(a_1),\ldots,(a_n)}
\sum_{\vec{x}\in\mathcal B_K^n}
\sum_{(\vec{x})}
\kappa_\pi^{h_{\mathbb G}}
(\vec a_{(1)})\cdot h_K(R_n(\vec{x}))
\Theta_{\vec{\gamma}_{(\vec{x})}}(\pi)
\nu_{\vec{x}}(\vec a_{(2)}).$$
Thus
\[
(h\ot\id)\Delta(\nu_{\vec{\gamma}}(\vec a))
=
\sum_{\pi\in\NC(n)}
\frac{1}{D^{n-|\pi|}}
F_\pi(\vec{\gamma},\vec a).
\]

\medskip

\medskip

\noindent For $V\in\pi$, write
$\pi\setminus V:=\bigcup_{W\in\pi,\ W\neq V}W$.
We view $\pi\setminus\{V\}$ as a non-crossing
partition on $\pi\setminus V$.

\medskip

\noindent\textbf{Claim.}
\textit{Let $\pi\in\NC(n)$ have at least two blocks
and let $V\in\pi$ be a consecutive block, i.e.
$V=\{a,a+1,\ldots,b\}$ for some
$1\leq a\leq b\leq n$. Then}
\[
\begin{aligned}
F_\pi(\vec{\gamma},\vec a)
={}&
h_K\bigl((\gamma_a\cdots\gamma_b)^*\bigr)
\kappa_{b-a+1}^{h_{\mathbb G}}(\vec a|_V)\cdot
F_{\pi\setminus\{V\}}
\bigl(
\vec{\gamma}|_{\pi\setminus V},
\vec a|_{\pi\setminus V}
\bigr),
\end{aligned}
\]

\medskip

\noindent\textit{Proof of the Claim.}
Write
\[
\{1,\ldots,n\}=V_-\sqcup V\sqcup V_+,
\]
where $V_-=\{1,\ldots,a-1\}$ and
$V_+=\{b+1,\ldots,n\}$.
For $\vec{x}=(x_1,\ldots,x_n)$, write
$\vec{\gamma}_{(\vec{x})}=(d_1,\ldots,d_n)$.

\noindent Note that
$$
\kappa_\pi^{h_{\mathbb G}}(\vec a_{(1)})
=
\kappa_{|V|}^{h_{\mathbb G}}(\vec a_{(1)}|_V)
\kappa_{\pi\setminus\{V\}}^{h_{\mathbb G}}
(\vec a_{(1)}|_{\pi\setminus V})\qquad
\Theta_{\vec{\gamma}_{(\vec{x})}}(\pi)
=
h_K\bigl((d_a\cdots d_b)^*\bigr)
\prod_{\substack{W=\{i_1<\cdots<i_s\}\\
W\in\pi\setminus\{V\}}}
h_K\bigl((d_{i_1}\cdots d_{i_s})^*\bigr).$$
We first simplify $(d_a\cdots d_b)^*$.

Since $K$ is a finite quantum group, it is
of Kac type. In particular, $S^2=\id$ and
\[
S(x^*)=S(x)^*,\qquad
\sum_{(x)}S(x_{(2)})x_{(1)}
=\varepsilon(x)1.
\]
Taking adjoints in the expression for $d_k$,
we obtain $$d_k^*
=x_{1(2k-1)}x_{2(2k-3)}
\cdots x_{k-1(3)}x_{k(1)}
\gamma_k^*\cdot
S(x_{k-1(2)})S(x_{k-2(4)})
\cdots S(x_{1(2k-2)}).$$

Write $d_i^*=P_i\gamma_i^*Q_i,$
where

$$P_i=x_{1(2i-1)}x_{2(2i-3)}
\cdots x_{i-1(3)}x_{i(1)},\quad
Q_i=S(x_{i-1(2)})S(x_{i-2(4)})
\cdots S(x_{1(2i-2)}).$$

Then
\[
\begin{aligned}
(d_a\cdots d_b)^*
={}&P_b\gamma_b^*
\underbrace{Q_bP_{b-1}}_{\text{contraction}}
\gamma_{b-1}^*
\underbrace{Q_{b-1}P_{b-2}}_{\text{contraction}}
\\
&\quad\cdots
\underbrace{Q_{a+1}P_a}_{\text{contraction}}
\gamma_a^*Q_a.
\end{aligned}
\]
There are exactly $b-a$ contraction regions,
each of the form
\[
\begin{aligned}
Q_{k+1}P_k
={}&S(x_{k(2)})\cdots S(x_{1(2k)})
x_{1(2k-1)}\cdots x_{k(1)},
\end{aligned}
\]

By the antipode identity, each region $Q_{k+1}P_k$
contracts successively from the innermost pair to the
outermost pair, eliminating two legs of each $x_j$,
$1\leq j\leq k$. Thus, in the full Sweedler summation,
the number of eliminated legs of $x_j$ is
\[
\begin{cases}
2(b-a), & j<a,\\
2(b-j), & a\leq j\leq b,\\
0, & j>b.
\end{cases}
\]
Indeed, $x_j$ occurs in all $b-a$ contraction regions
if $j<a$, and in the $b-j$ regions $Q_{k+1}P_k$,
$j\leq k<b$, if $a\leq j\leq b$.

\noindent After these contractions and the corresponding
renumbering of the Sweedler indices, $(d_a\cdots d_b)^*$
becomes
\[
\begin{aligned}
&x_{1(2a-1)}\cdots x_{a-1(3)}
\underbrace{x_{a(1)}\cdots x_{b(1)}}
(\gamma_a\cdots\gamma_b)^*\\
&\qquad\cdot
S(x_{a-1(2)})\cdots S(x_{1(2a-2)}).
\end{aligned}
\]

For $i<a$, none of $x_a,\ldots,x_b$ occurs in $d_i^*$.
For $i>b$ and $a\leq j\leq b$, the two indices of $x_j$
in $d_i^*$ are shifted as follows:
\[
2(i-j)+1\longmapsto 2(i-b)+1,\qquad
2(i-j)\longmapsto 2(i-b).
\]
Similarly, its index in $R_n(\vec{x})$ becomes
\[
2(n-j+1)\longmapsto 2(n-b+1).
\]
Thus, after the contractions, the remaining Sweedler
legs of $x_a,\ldots,x_b$ in
$h_K(R_n(\vec{x}))\Theta_{\vec{\gamma}_{(\vec{x})}}(\pi)$ form the following consecutive products, where
$X_V:=x_a\cdots x_b$:
\[
\begin{array}{ll}
x_{a(1)}\cdots x_{b(1)}=(X_V)_{(1)}
&\text{in }(d_a\cdots d_b)^*,\\[3pt]
x_{a(2(i-b)+1)}\cdots x_{b(2(i-b)+1)}
=(X_V)_{(2(i-b)+1)}
&\text{in }d_i^*,\quad i>b,\\[3pt]
S(x_{b(2(i-b))})\cdots S(x_{a(2(i-b))})
=S((X_V)_{(2(i-b))})
&\text{in }d_i^*,\quad i>b,\\[3pt]
x_{a(2(n-b+1))}\cdots x_{b(2(n-b+1))}
=(X_V)_{(2(n-b+1))}
&\text{in }R_n(\vec{x}).
\end{array}
\]

\medskip

\noindent Fix $\vec{x}_-=(x_1,\ldots,x_{a-1})$
and $\vec{x}_+=(x_{b+1},\ldots,x_n)$.
For $z\in C(K)$, define
\[
\begin{aligned}
\varphi_{\vec{x}_+,\vec{x}_-}(z)
={}&\sum_{(\vec{x}_-),(\vec{x}_+),(z)}
h_K\left(
\prod_{j<a}x_{j(2(n-b+a-j+1))}
z_{(2(n-b+1))}
\prod_{j>b}x_{j(2(n-j+1))}
\right)\\
&\cdot h_K\left(
\prod_{j<a}x_{j(2(a-j)+1)}
z_{(1)}(\gamma_a\cdots\gamma_b)^*
\overleftarrow{\prod}_{j<a}S(x_{j(2(a-j))})
\right)\\
&\cdot\prod_{W\in\pi\setminus\{V\}}h_K\Bigg(
\overleftarrow{\prod}_{\substack{i\in W\\i>b}}
\Bigg[
\prod_{j<a}x_{j(2(i-b+a-j)+1)}
z_{(2(i-b)+1)}
\prod_{b<j\leq i}x_{j(2(i-j)+1)}
\gamma_i^*\\
&\hspace{3cm}\cdot
\overleftarrow{\prod}_{b<j<i}S(x_{j(2(i-j))})
S(z_{(2(i-b))})
\overleftarrow{\prod}_{j<a}S(x_{j(2(i-b+a-j))})
\Bigg]
\overleftarrow{\prod}_{\substack{i\in W\\i<a}}d_i^*
\Bigg),
\end{aligned}
\]
where $\overleftarrow{\prod}$ denotes the product
in decreasing order.

Since $\vec{x}_-$ and $\vec{x}_+$ are fixed and
each of the $2(n-b+1)$ legs of the iterated
coproduct of $z$ occurs exactly once,
the linearity of $\Delta_K$, $S$ and $h_K$
implies that $\varphi_{\vec{x}_+,\vec{x}_-}$
is a linear functional on $C(K)$.
By the preceding computations,
\[
\sum_{(\vec{x})}
h_K(R_n(\vec{x}))
\Theta_{\vec{\gamma}_{(\vec{x})}}(\pi)
=
\varphi_{\vec{x}_+,\vec{x}_-}(x_a\cdots x_b).
\]

\noindent For fixed $\vec{x}_+$ and $\vec{x}_-$,
the contribution of the variables corresponding to $V$ is
\[
\begin{aligned}
&\sum_{(a_a),\ldots,(a_b)}
\kappa_{|V|}^{h_{\mathbb G}}(\vec a_{(1)}|_V)
\sum_{x_a,\ldots,x_b\in\mathcal B_K}
\varphi_{\vec{x}_+,\vec{x}_-}(x_a\cdots x_b)
\nu_{x_a}(a_{a(2)})\cdots\nu_{x_b}(a_{b(2)})\\
&=\sum_{(a_a),\ldots,(a_b)}
\kappa_{|V|}^{h_{\mathbb G}}(\vec a_{(1)}|_V)
(\varphi_{\vec{x}_+,\vec{x}_-}\ot\id)
\rho(a_{a(2)}\cdots a_{b(2)})\\
&=(\varphi_{\vec{x}_+,\vec{x}_-}\ot\id)
\rho\left(
\sum_{(a_a),\ldots,(a_b)}
\kappa_{|V|}^{h_{\mathbb G}}(\vec a_{(1)}|_V)
a_{a(2)}\cdots a_{b(2)}
\right)\\
&=\kappa_{|V|}^{h_{\mathbb G}}(\vec a|_V)
\varphi_{\vec{x}_+,\vec{x}_-}(1)1,
\end{aligned}
\]
where we used the multiplicativity of $\rho$,
the linearity of $(\varphi_{\vec{x}_+,\vec{x}_-}
\ot\id)\circ\rho$, and Lemma \ref{LemCumulantHaar}.

Consequently,
$$F_\pi(\vec{\gamma},\vec a)
=\kappa_{|V|}^{h_{\mathbb G}}(\vec a|_V)
\sum_{\vec{x}_-,\vec{x}_+}
\sum_{(a_i):\,i\notin V}
\kappa_{\pi\setminus\{V\}}^{h_{\mathbb G}}
(\vec a_{(1)}|_{\pi\setminus V})\cdot
\varphi_{\vec{x}_+,\vec{x}_-}(1)
\nu_{\vec{x}_-}(\vec a_{(2)}|_{V_-})
\nu_{\vec{x}_+}(\vec a_{(2)}|_{V_+}).$$

\noindent To compute $\varphi_{\vec{x}_+,\vec{x}_-}(1)$,
note that $\Delta_K(1)=1\otimes1$, so all the legs of $z$
disappear. Let $r(i)=i$ for $i<a$ and $r(i)=i-|V|$ for $i>b$
be the increasing relabelling of $\pi\setminus V$.
For $j<a$ and $i>b$, the remaining indices of $x_j$ in
$d_i^*$ are
\[
2(i-j)-2(|V|-1)=2(r(i)-r(j))+2,\qquad
2(r(i)-r(j))+3,
\]
and its index in $R_n(\vec{x})$ is $2(n-|V|-r(j)+1)+2$.
Thus these indices exceed the required indices after
deleting $V$ by two.

By the traciality of $h_K$, the factor corresponding to
$V$ becomes
\[
h_K\bigl(
(\gamma_a\cdots\gamma_b)^*
\underbrace{
S(x_{a-1(2)})\cdots S(x_{1(2a-2)})
x_{1(2a-1)}\cdots x_{a-1(3)}
}_{\text{antipode contractions}}
\bigr).
\]
In the full Sweedler summation, these contractions yield
$h_K((\gamma_a\cdots\gamma_b)^*)$ and eliminate the two
legs of each $x_j$, $j<a$, with indices $2(a-j)$ and
$2(a-j)+1$. Hence the above indices decrease by two,
while all other remaining indices are unchanged.

Writing $\pi'=\pi\setminus\{V\}$,
$\vec{x}'=\vec{x}|_{\pi\setminus V}$ and
$\vec{\gamma}'=\vec{\gamma}|_{\pi\setminus V}$,
we obtain
\[
\varphi_{\vec{x}_+,\vec{x}_-}(1)
=h_K((\gamma_a\cdots\gamma_b)^*)
\sum_{(\vec{x}')}
h_K(R_{n-|V|}(\vec{x}'))
\Theta_{\vec{\gamma}'_{(\vec{x}')}}(\pi').
\]
Consequently,
\[
F_\pi(\vec{\gamma},\vec a)
=h_K((\gamma_a\cdots\gamma_b)^*)
\kappa_{|V|}^{h_{\mathbb G}}(\vec a|_V)
F_{\pi'}(\vec{\gamma}',\vec a|_{\pi\setminus V}),
\]
which proves the Claim.

\medskip

\noindent\textit{End of the Proof of Theorem \ref{ThmHaar}.}
By the Claim, it suffices to prove
\[
F_\pi(\vec{\gamma},\vec a)
=\Theta_{\vec{\gamma}}(\pi)
\kappa_\pi^{h_{\mathbb G}}(\vec a)1
\]
for $\pi=1_n$. By the preceding contractions,
putting $X=x_1\cdots x_n$, we have
\[
\begin{aligned}
\sum_{(\vec{x})}
h_K(R_n(\vec{x}))
h_K((d_1\cdots d_n)^*)
&=\sum_{(X)}
h_K(X_{(2)})
h_K(X_{(1)}(\gamma_1\cdots\gamma_n)^*)\\
&=h_K(X)h_K((\gamma_1\cdots\gamma_n)^*).
\end{aligned}
\]
It follows that
\[
\begin{aligned}
F_{1_n}(\vec{\gamma},\vec a)
&=h_K((\gamma_1\cdots\gamma_n)^*)
\sum_{(a_1),\ldots,(a_n)}
\kappa_n^{h_{\mathbb G}}(\vec a_{(1)})
(h_K\ot\id)\rho(a_{1(2)}\cdots a_{n(2)})\\
&=h_K((\gamma_1\cdots\gamma_n)^*)
(h_K\ot\id)\rho\left(
\sum_{(a_1),\ldots,(a_n)}
\kappa_n^{h_{\mathbb G}}(\vec a_{(1)})
a_{1(2)}\cdots a_{n(2)}
\right)\\
&=h_K((\gamma_1\cdots\gamma_n)^*)
\kappa_n^{h_{\mathbb G}}(\vec a)1\\
&=\Theta_{\vec{\gamma}}(1_n)
\kappa_n^{h_{\mathbb G}}(\vec a)1,
\end{aligned}
\]
where we used the multiplicativity of $\rho$
and Lemma \ref{LemCumulantHaar}.

Hence, by induction on $|\pi|$, the asserted identity
holds for all $\pi\in\NC(n)$. Consequently,
\[
(h\ot\id)\Delta(\nu_{\vec{\gamma}}(\vec a))
=
\sum_{\pi\in\NC(n)}
\frac{\Theta_{\vec{\gamma}}(\pi)
\kappa_\pi^{h_{\mathbb G}}(\vec a)}
{D^{n-|\pi|}}1
=h(\nu_{\vec{\gamma}}(\vec a))1.
\]

Finally, for $b\in C(K)$, the same computation with
$h_K(R_n(\vec x)b_{(1)})$ in place of $h_K(R_n(\vec x))$
yields, by the Haar invariance of $h_K$,

$$(h\otimes\id)\Delta(\nu_{\vec\gamma}(\vec a)b)=h(\nu_{\vec\gamma}(\vec a))
\sum_{(b)}h_K(b_{(1)})b_{(2)}=h(\nu_{\vec\gamma}(\vec a)b)1.$$

By density and uniqueness of the invariant state,
$h$ is the Haar state of $\mathbb W$.

\end{proof}

\noindent For a group $\Lambda$ and $\vec{\gamma}\in\Lambda^n$,
denote by $\NC(\vec{\gamma})$ the set of $\pi\in\NC(n)$
such that $\gamma_{r_1}\cdots\gamma_{r_s}=1$ for every
block $V=\{r_1<\cdots<r_s\}\in\pi$.
For $\pi\in\NC(\vec{\gamma})$, put
\[
\mu_\pi(\vec{\gamma})
:=\prod_{V\in\pi}
\sum_{\sigma\in\NC(\vec{\gamma}|_V)}
\mu(\sigma,1_{|V|}),
\]
where $\mu$ denotes the M\"obius function on non-crossing
partitions; see \cite{nica2006lectures}.

\begin{corollary}
Let $\Lambda$ be a finite group and $K=\widehat{\Lambda}$.
For all $\vec{\gamma}\in\Lambda^n$,
$\vec a\in\Pol(\mathbb G)^n$ and $s\in\Lambda$, one has
\[
h(\nu_{\vec{\gamma}}(\vec a)s)
=\delta_{s,1}
\sum_{\pi\in\NC(\vec{\gamma})}
\frac{\kappa_\pi^{h_{\mathbb G}}(\vec a)}
{|\Lambda|^{n-|\pi|}}.
\]
In particular, if $\mathbb G=\widehat{\Gamma}$ for a
countable discrete group $\Gamma$, then, for all
$\vec g\in\Gamma^n$,
\[
h(\nu_{\vec{\gamma}}(\vec g)s)
=\delta_{s,1}
\sum_{\pi\in\NC(\vec g)\cap\NC(\vec{\gamma})}
\frac{\mu_\pi(\vec g)}{|\Lambda|^{n-|\pi|}}.
\]
\end{corollary}

\noindent Theorem~\ref{ThmUnivState} shows that $\widetilde\omega$ is $\alpha$-invariant. Since $h=\widetilde\omega\circ E$, its GNS representation is the regular crossed-product representation, and we obtain the following Corollary.

\begin{corollary}\label{CorCrossed}
The reduced C*-algebra and the von Neumann algebra of $\mathbb W$ are:
$$C_r(\mathbb W)=\Ucal_r\rtimes_r\Lambda\text{ and }\Linf(\mathbb W)=\Ucal''\rtimes\Lambda.$$
\end{corollary}

\section{Operator algebras of the free wreath product}

\noindent In this section we study the properties of the reduced C*-algebra $C_r(\mathbb W)$ and the von Neumann algebra $\Linf(\mathbb W)$.

\subsection{The von Neumann algebra}

\begin{theorem}\label{ThmVN}
For $\chi\in\widehat{\mathcal Z(\Lambda)}$, set $p_\chi
=
|\mathcal Z(\Lambda)|^{-1}
\sum_{z\in\mathcal Z(\Lambda)}
\overline{\chi(z)}u_z
\in \Linf(\mathbb W).$
Then $\Linf(\mathbb W)
=\bigoplus_{\chi\in\widehat{\mathcal Z(\Lambda)}}
\Linf(\mathbb W)p_\chi.$
Moreover, for every section $s:\Lambda/\mathcal Z(\Lambda)\to\Lambda$
with $s(e)=e$, one has $\Linf(\mathbb W)p_\chi
\cong\Ucal''\rtimes_{\bar\alpha,c_\chi}\Lambda/\mathcal Z(\Lambda),$
where $\bar\alpha:\Lambda/\mathcal Z(\Lambda)\curvearrowright \Ucal''$ is the action induced by $\alpha$, i.e. $\bar\alpha_{\pi(r)}=\alpha_r, r\in\Lambda,$
where $\pi:\Lambda\to\Lambda/\mathcal Z(\Lambda)$ is the quotient map,
and $c_\chi(q,r)
=\chi\bigl(s(q)s(r)s(qr)^{-1}\bigr),q,r\in\Lambda/\mathcal Z(\Lambda).$

\medskip

\noindent If $\mathbb G$ is of Kac type, $\Linf(\mathbb G)$ is a diffuse factor and $\Lambda$ is non-trivial, then the above direct sum is
the central decomposition of $\Linf(\mathbb W)$, and each summand $\Linf(\mathbb W)p_\chi$
is a full $\mathrm{II}_1$-factor.
\end{theorem}

\begin{proof}
We use the identification $\Linf(\mathbb W)\cong \Ucal''\rtimes_\alpha\Lambda.$
Since $\mathcal Z(\Lambda)\subset\ker(\Lambda\curvearrowright\Ucal'')$,
Lemma~\ref{LemCentralSummandTwistedCrossedProduct} gives $\Linf(\mathbb W)
=\bigoplus_{\chi\in\widehat{\mathcal Z(\Lambda)}}
\Linf(\mathbb W)p_\chi$
and, for every section $s:\Lambda/\mathcal Z(\Lambda)\to\Lambda$ with $s(e)=e$,
an isomorphism $\Linf(\mathbb W)p_\chi\cong
\Ucal''\rtimes_{\bar\alpha,c_\chi}\Lambda/\mathcal Z(\Lambda),$
with $\bar\alpha$ and $c_\chi$ as in the statement.
Assume now that $\mathbb G$ is of Kac type, $\Linf(\mathbb G)$ is a diffuse factor and $\Lambda$ is non-trivial. We show that
the induced action $\Lambda/\mathcal Z(\Lambda)\curvearrowright\Ucal''$
is outer. Let $r\in\Lambda$, and assume that there exists a unitary
$w\in\Ucal''$ such that $\alpha_r=\Ad(w).$
Then, for every $a\in\Linf(\mathbb G)$,
$$(\id\otimes\alpha_r)\rho''(a)
=
(\Ad(r^{-1})\otimes\id)\rho''(a).$$
Thus $r\otimes w
\in
\rho''(\Linf(\mathbb G))'
\cap
\bigl(C^*(\Lambda)\otimes\Ucal''\bigr).$
Since $\Linf(\mathbb G)$ is a diffuse factor, Lemma~\ref{LemFactor} gives $r\otimes w\in \mathcal Z(C^*(\Lambda))\otimes 1.$
Hence $w\in\mathbb C1$
and $r\in\mathcal Z(\Lambda)$.
Therefore $\alpha_r \text{ is outer for every }r\in\Lambda\setminus\mathcal Z(\Lambda).$

\medskip

\noindent By Corollary~\ref{CorTwistedSummandsIIoneFactors}, the above direct sum is the
central decomposition of $\Linf(\mathbb W)$, and each summand $\Linf(\mathbb W)p_\chi$
is a $\mathrm{II}_1$-factor. Finally, by \cite[Theorem~C]{FMP24}, $\Ucal''$ is a full
$\mathrm{II}_1$-factor. Since the induced action $\Lambda/\mathcal Z(\Lambda)\curvearrowright\Ucal''$
is outer, Corollary~\ref{CorTwistedCrossedProductFull} applied to $\Ucal''\rtimes_{\bar\alpha,c_\chi}\Lambda/\mathcal Z(\Lambda)$
shows that each $\Linf(\mathbb W)p_\chi$
is full.
\end{proof}

\begin{corollary}
If $\mathbb G$ is of Kac type, $\Linf(\mathbb G)$ is a diffuse factor, $\Lambda$ is non-trivial and has trivial center, then $\Linf(\mathbb W)$ is a full ${\rm II}_1$-factor.
\end{corollary}

\subsection{The reduced C*-algebra}

\begin{theorem}\label{ThmCstar}
For $\chi\in\widehat{\mathcal Z(\Lambda)}$, set
$p_\chi=|\mathcal Z(\Lambda)|^{-1}\sum_{z\in\mathcal Z(\Lambda)}\overline{\chi(z)}u_z\in C_r(\mathbb W).$
Then

$$
C_r(\mathbb W)=\bigoplus_{\chi\in\widehat{\mathcal Z(\Lambda)}}C_r(\mathbb W)p_\chi.
$$

Moreover, for every section $s:\Lambda/\mathcal Z(\Lambda)\to\Lambda$ with $s(e)=e$, one has

$$
C_r(\mathbb W)p_\chi\cong \Ucal_r\rtimes_{r,\bar\alpha,c_\chi}\Lambda/\mathcal Z(\Lambda),
$$

where $\bar\alpha:\Lambda/\mathcal Z(\Lambda)\curvearrowright \Ucal_r$ is the action induced by $\alpha$, that is, $\bar\alpha_{\pi(r)}=\alpha_r$ for all $r\in\Lambda$, with $\pi:\Lambda\to\Lambda/\mathcal Z(\Lambda)$ the quotient map, and
$c_\chi(q,r)=\chi\bigl(s(q)s(r)s(qr)^{-1}\bigr)$, $q,r\in\Lambda/\mathcal Z(\Lambda)$.

\medskip

\noindent If $\mathbb G$ is of Kac type, $\Linf(\mathbb G)$ is a diffuse factor and $\Lambda$ is non-trivial, then each summand $C_r(\mathbb W)p_\chi$ is simple with a unique trace.
\end{theorem}

\begin{proof}
We use the identification $C_r(\mathbb W)\cong \Ucal_r\rtimes_{r,\alpha}\Lambda.$ Since $\mathcal Z(\Lambda)\subset\ker(\Lambda\curvearrowright\Ucal_r)$, the projections
$p_\chi=|\mathcal Z(\Lambda)|^{-1}\sum_{z\in\mathcal Z(\Lambda)}\overline{\chi(z)}u_z$, $\chi\in\widehat{\mathcal Z(\Lambda)}$, are central in $C_r(\mathbb W)$, mutually orthogonal, and sum to $1$. Hence

$$
C_r(\mathbb W)=\bigoplus_{\chi\in\widehat{\mathcal Z(\Lambda)}}C_r(\mathbb W)p_\chi.
$$

By the $C^*$-algebraic version of Lemma~\ref{LemCentralSummandTwistedCrossedProduct} (see Remark~\ref{RmkCrossed}), for every section $s:\Lambda/\mathcal Z(\Lambda)\to\Lambda$ with $s(e)=e$, one has

$$
C_r(\mathbb W)p_\chi\cong \Ucal_r\rtimes_{r,\bar\alpha,c_\chi}\Lambda/\mathcal Z(\Lambda),
$$

with $\bar\alpha$ and $c_\chi$ as in the statement.

Assume now that $\mathbb G$ is of Kac type, that $\Linf(\mathbb G)$ is a diffuse factor and that $\Lambda$ is non-trivial. By the proof of Theorem~\ref{ThmVN}, the induced action $\Lambda/\mathcal Z(\Lambda)\curvearrowright \Ucal''$ is outer. Moreover, by \cite[Theorem~A]{FMP24}, $\Ucal_r$ is simple with a unique trace $\widetilde\omega$. Since this trace is unique, it is $\bar\alpha$-invariant, and its GNS von Neumann algebra is $\Ucal''$. Therefore, by \cite[Theorem~2]{Bedos96} applied to the reduced twisted crossed product, each algebra
$\Ucal_r\rtimes_{r,\bar\alpha,c_\chi}\Lambda/\mathcal Z(\Lambda)$
is simple with a unique trace. Hence each summand $C_r(\mathbb W)p_\chi$ is simple with a unique trace.
\end{proof}

\begin{corollary}
If $\mathbb G$ is of Kac type, $\Linf(\mathbb G)$ is a diffuse factor, and $\Lambda$ is non-trivial with trivial center, then $C_r(\mathbb W)$ is simple with a unique trace.
\end{corollary}

\begin{remark}
The center decomposition also gives a way to distinguish the resulting compact quantum groups. Indeed, fix a compact quantum group $\mathbb G$ of Kac type such that $\Linf(\mathbb G)$ is a diffuse factor. If $\Lambda_1$ and $\Lambda_2$ are finite non-trivial groups such that $|\mathcal Z(\Lambda_1)|\neq |\mathcal Z(\Lambda_2)|$, then the centers of the corresponding von Neumann algebras have different dimensions:

$$
\dim\mathcal Z\bigl(\Linf(\mathbb W(\mathbb G,\Lambda_i))\bigr)=|\mathcal Z(\Lambda_i)|,\qquad i=1,2.
$$

Hence these von Neumann algebras are not isomorphic. In particular, the associated compact quantum groups $\mathbb W(\mathbb G,\Lambda_1)$ and $\mathbb W(\mathbb G,\Lambda_2)$ are not isomorphic.
\end{remark}

\begin{remark}
The non-Kac case leads to additional modular and type ${\rm III}$ phenomena and will be investigated separately.
\end{remark}

\appendix

\section{Free wreath products by quantum automorphism groups}

\noindent We show now that $\GG$ can also be described as a free wreath product by a quantum automorphism group a finite graph in the sense of \cite{FP16}.

We first relate the free wreath products associated with finite quantum graphs introduced in \cite{FP16} to the generalized free wreath products of \cite{FT25}. Let $(B,\psi,d)$ be a finite quantum graph. By \cite[Proposition~2.3]{FP16}, its quantum automorphism group $\operatorname{Aut}^{+}(B,\psi,d)$ is obtained from $\operatorname{Aut}^{+}(B,\psi)$ by imposing the additional intertwining relation $(d\otimes1)u=u(d\otimes1)$. Fima and Pittau then define in \cite[Definition~2.6]{FP16} the corresponding free wreath product

$$
\mathbb G\wr_{*}\operatorname{Aut}^{+}(B,\psi,d).
$$

When $d=1$, \cite[Proposition~3.12]{FT25} identifies the Fima--Pittau construction with the generalized free wreath product associated with the canonical action of $\operatorname{Aut}^{+}(B,\psi)$ on $(B,\psi)$. The following proposition extends this identification to an arbitrary adjacency operator $d$.

\begin{proposition}
\label{prop:FT-d}
Let \(\mathbb G\) be a compact quantum group and let \((B,\psi,d)\) be a finite
quantum graph. Let
\[
\beta_d:B\longrightarrow
B\otimes C\bigl(\operatorname{Aut}^{+}(B,\psi,d)\bigr)
\]
be the canonical action. Then $\mathbb G\wr_{*}\operatorname{Aut}^{+}(B,\psi,d)
\;\simeq\;
\mathbb G\wr_{*,\beta_d}\operatorname{Aut}^{+}(B,\psi,d).$

Here the left-hand side is the free wreath product of \cite[Definition~2.3]{FP16}, with $C(\operatorname{Aut}^{+}(B,\psi,d))$ obtained from $C(\operatorname{Aut}^{+}(B,\psi))$ by imposing $(d\otimes1)u=u(d\otimes1)$, while the right-hand side is the generalized free wreath product of \cite{FT25}. For $d=1$, this is exactly \cite[Proposition~3.12]{FT25}.
\end{proposition}

\begin{proof}

Set
$\mathbb H_0:=\operatorname{Aut}^+(B,\psi), \mathbb H_d:=\operatorname{Aut}^+(B,\psi,d),A_0:=C(\mathbb G\wr_{*,\beta_0}\mathbb H_0), A_d:=C(\mathbb G\wr_{*,\beta_d}\mathbb H_d).$ and write
\[
A^{\mathrm{FP}}_{0}:=C(\mathbb G\wr_*\mathbb H_0), A^{\mathrm{FP}}_{d}:=C(\mathbb G\wr_*\mathbb H_d)
:=A^{\mathrm{FP}}_{0}\Big/
\Big\langle\operatorname{coeff}\big((d\otimes1)r(1)-r(1)(d\otimes1)\big)\Big\rangle
\quad\text{\cite[Definition~2.3]{FP16}},
\]Let $\Phi_0:A^{\mathrm{FP}}_0\to A_0$ be the isomorphism of \cite[Proposition~3.12]{FT25}. Since $(\id\otimes\Phi_0)(r(1))=(\id\otimes\iota_0)(u^{(0)})$, $\Phi_0$ sends the defining ideal of $A^{\mathrm{FP}}_d$ onto $J_d:=\langle\iota_0(\ker q)\rangle_{A_0}$, where $q:C(\mathbb H_0)\to C(\mathbb H_d)$ is the canonical quotient. Hence $A^{\mathrm{FP}}_d\simeq A_0/J_d$. It remains to prove that $A_0/J_d\simeq A_d$. 

\medskip

\noindent Let $q:C(H_0)\to C(H_d)$ be the quotient map and put
$\beta_d=(\id\otimes q)\beta_0$.  Since $[(\id\otimes\iota_d)(\id\otimes q)\beta_0(b),\rho_d(a)]
=[(\id\otimes\iota_d)\beta_d(b),\rho_d(a)]=0$
for all $a\in C(G)$ and $b\in B$, the universal property yields a
$*$-homomorphism $\Pi:C(G\wr_{*,\beta_0}H_0)\to C(G\wr_{*,\beta_d}H_d)$
such that $\Pi\circ\iota_0=\iota_d\circ q$. Hence, for every
$x\in\ker q$, $\Pi(\iota_0(x))=\iota_d(q(x))=0.$
Thus $\iota_0(\ker q)\subseteq\ker\Pi$, and since $\ker\Pi$ is a closed
two-sided ideal, $J_d\subseteq\ker\Pi.$
Therefore $\Pi$ factors through
the quotient, giving a $*$-homomorphism
\[
\widetilde{\Pi}:
C(G\wr_{*,\beta_0}H_0)/J_d
\longrightarrow
C(G\wr_{*,\beta_d}H_d)
\]
such that $\widetilde{\Pi}\circ p=\Pi$, where
$p:C(G\wr_{*,\beta_0}H_0)\to C(G\wr_{*,\beta_0}H_0)/J_d$
is the quotient map.

Conversely, since $p\circ\iota_0$ vanishes on $\ker q$, there is a unique
$*$-homomorphism $\bar\iota_d:C(H_d)\to C(G\wr_{*,\beta_0}H_0)/J_d$
such that $\bar\iota_d\circ q=p\circ\iota_0$. Put
$\bar\rho=(\id\otimes p)\rho_0$. Then, for all $a\in C(G)$ and $b\in B$,
\[
\begin{aligned}
[(\id\otimes\bar\iota_d)\beta_d(b),\bar\rho(a)]
&=[(\id\otimes\bar\iota_d)(\id\otimes q)\beta_0(b),
   (\id\otimes p)\rho_0(a)]
=(\id\otimes p)
  [(\id\otimes\iota_0)\beta_0(b),\rho_0(a)]
=0.
\end{aligned}
\]
Hence, by the universal property, there exists a $*$-homomorphism
\[
\Psi:C(G\wr_{*,\beta_d}H_d)
\longrightarrow C(G\wr_{*,\beta_0}H_0)/J_d
\]
such that $\Psi\circ\iota_d=\bar\iota_d$ and
$(\id\otimes\Psi)\rho_d=\bar\rho$.
The maps $\Psi$ and $\widetilde{\Pi}$ are inverse on the canonical
generators, and therefore
\[
C(G\wr_{*,\beta_0}H_0)/J_d
\simeq C(G\wr_{*,\beta_d}H_d).
\]

\end{proof}

\medskip

\noindent We now specialize the preceding correspondence to the quantum Cayley graphs associated with finite group duals. Let $\Lambda$ be a non-abelian finite group and let $d_\Lambda$ be an adjacency operator such that $\operatorname{Aut}^+(C^*(\Lambda),\tau_\Lambda,d_\Lambda)\simeq\widehat{\Lambda}$, as in \cite{BGMSW25}. Proposition~\ref{prop:FT-d} then identifies the generalized free wreath product by the canonical action of $\widehat{\Lambda}$ on $C^*(\Lambda)$ with the corresponding Fima--Pittau free wreath product associated with the quantum graph $(C^*(\Lambda),\tau_\Lambda,d_\Lambda)$.

\bibliography{ref}
\bibliographystyle{amsalpha}

\end{document}